\documentclass[a4paper,11pt]{article}
\usepackage{courier}
\usepackage{blindtext} 
\usepackage{amsmath}
\usepackage{amssymb}
\usepackage[colorlinks=true,breaklinks]{hyperref}
\usepackage[hyphenbreaks]{breakurl}
\usepackage{xcolor}
\definecolor{c1}{rgb}{0,0,1}
\definecolor{c2}{rgb}{0,0.3,0.9}
\definecolor{c3}{rgb}{0.3,0.9}
\hypersetup{linkcolor={c1},citecolor={c2},urlcolor={c3}}
\usepackage{enumerate}
\usepackage{todonotes}
\usepackage{makeidx}
\makeindex
\usepackage[top=4cm,bottom=4cm,left=2.8cm,right=2.8cm]{geometry}
\usepackage{fancyhdr}

\newcommand{\bbR}{\mathbb{R}}
\newcommand{\R}{\bbR}

\def\XXint#1#2#3{{\setbox0=\hbox{$#1{#2#3}{\int}$ }
\vcenter{\hbox{$#2#3$ }}\kern-.6\wd0}}

\usepackage{amsthm}
\theoremstyle{plain}
\newtheorem{theorem}{Theorem}[section]

\theoremstyle{definition}

\theoremstyle{lemma}
\newtheorem{lemma}[theorem]{Lemma}

\theoremstyle{remark}

\theoremstyle{proposition}
\newtheorem{proposition}[theorem]{Proposition}

\theoremstyle{corollary}
\newtheorem{corollary}[theorem]{Corollary}

\usepackage{systeme}
\usepackage{colonequals}

\begin{document}
\pagestyle{empty}
\title{Strong $L^2 H^2$ convergence of the JKO scheme\\ for the Fokker-Planck equation}
\author{Filippo Santambrogio\thanks{Institut Camille Jordan, Universit\'e Claude Bernard Lyon 1,   69622 Villeurbanne cedex, France, {\tt santambrogio@math.univ-lyon1.fr}}, Gayrat Toshpulatov\thanks{Institut f\"ur Analysis \& Scientific Computing,Technische Universit\"at Wien, Wiedner Hauptstr. 8, 1040 Wien, Austria, {\tt gayrat.toshpulatov@tuwien.ac.at}}}

\maketitle

\pagestyle{plain}
\begin{abstract}
Following a celebrated paper by Jordan, Kinderleherer and Otto it is possible to discretize in time the Fokker-Planck equation $\partial_t\varrho=\Delta\varrho+\nabla\cdot(\rho\nabla V)$ by solving a sequence of iterated variational problems in the Wasserstein space, and the sequence of piecewise constant curves obtained from the scheme is known to converge to the solution of the continuous PDE. This convergence is uniform in time valued in the Wasserstein space and also strong in $L^1$ in space-time. We prove in this paper, under some assumptions on the domain (a bounded and smooth convex domain) and on the initial datum (which is supposed to be bounded away from zero and infinity and belong to $W^{1,p}$ for an exponent $p$ larger than the dimension), that the convergence is actually strong in $L^2_tH^2_x$, hence strongly improving the previously known results in terms of the order of derivation in space. The technique is based on some inequalities, obtained with optimal transport techniques, that can be proven on the discrete sequence of approximate solutions, and that mimic the corresponding continuous computations.
\end{abstract}
\tableofcontents

\section{Introduction}

More than 20 years ago Jordan, Kinderleherer and Otto wrote their seminal paper \cite{JKO}, where they identified a variational structure in the Fokker-Plank equation 
$$\partial_t\varrho=\Delta\varrho+\nabla\cdot(\rho\nabla V)$$
as a gradient flow of the functional $J(\varrho)=\int \varrho(V+\log\varrho)$ in the Wasserstein space $W_2$. This is remarkable since the same equation has no gradient-flow structure for Hilbertian distances such as the $L^2$ or $H^{-1}$ norms, differently from the heat equation (obtained for $V=0$). The gradient flow interpretation also gives a natural discretization in time, where a time step $\tau>0$ is fixed and a sequence $(\varrho^\tau_k)_k$ is iteratively defined using
$$\varrho_{k+1}^{\tau} \in \text{argmin}_{\varrho}\left\{ J(\varrho)+\dfrac{{W}^2_2(\varrho,\varrho_k^{\tau})}{2\tau}\right\}.
$$
This sequence is then used to define a curve $t\mapsto \varrho^{\tau}(t)$ in the space of probability measures via $\varrho^{\tau}(0)=\varrho_0$ and
$$\varrho^{\tau}(t)=\varrho_{k+1}^{\tau}  \text{   for  } t\in (k\tau,(k+1)\tau].$$

In  \cite{JKO} the convergence of $\varrho^\tau$ to the solution of the Fokker-Planck equation was proven when the domain on which the equation is set is the whole space (of course under suitable decay assumptions on the initial datum, namely that it has finite second moment $\int |x|^2d\varrho_0(x)<+\infty),$ and the convergence is weak in $L^1$ for every $t$ and strong in $L^1([0,T]\times\R^d)$. The analysis of the convergence can be adapted to the case of a bounded domain $\Omega$ (and in this case there is no moment condition) and the results are essentially the same. This is, for instance, the object of Chapter 8 in \cite{OTAM}, where the Fokker-Plank equation is chosen as an example to present the JKO scheme and the gradient-flow approach to some diffusion equations. As it is a linear equation, it is also the simplest case where this analysis can be performed.

The goal of the present paper is to improve the nature of the above convergence, under some possible assumptions on the initial datum. We do not mean obtaining a better rate of convergence in terms of $\tau$ (note that \cite{AGS} proves a convergence of order $O(\tau)$ in the Wasserstein distance $W_2$ whenever $V$ is semi-convex), but obtaining strong convergences in the best possible spaces.

The paper will undergo these proofs of convergence by analyzing different steps. The first one consists in proving a very classical $L^2_tH^1_x$ bound on the discrete solutions $\varrho^\tau$, which is obtained by a discrete analogous of a very well-known computation : the time derivative of $\int \frac12\varrho^2$ equals $-\int |\nabla\varrho|^2+\varrho \nabla\varrho\cdot \nabla V$ whenever $\varrho$ solves the Fokker-Planck equation. A similar computation can be done for the sequence obtained via the JKO scheme, but strongly relies on the geodesic convexity of the functional $\varrho\mapsto\int \frac12\varrho^2$. This estimate provides strong $L^2$ compactness in space, and allows to obtain convergence in $L^2_tL^2_x$ via the Aubin-Lions lemma when coupled with bounds in time, which are obtained via a suitable interpolation which allows to apply one of the most classical versions of Aubin and Lions' result.  The strong $L^2_tL^2_x$ convergence is not surprising, and it is just a small refinement of the original $L^1$ convergence already proven by Jordan-Kinderlehrer-Otto, but is a necessary step to go on. The next step consists the strong convergence in $L^2_tH^1_x$. This is obtained by refining the same computations. Once we have a bound on $\int_0^T\int_\Omega |\nabla\varrho^\tau|^2$, this provides weak convergence in $L^2_tH^1_x$ and the limit can only be the solution $\varrho$ to the limit Fokker-Planck equation. We do have strong convergence if we are able to prove $\limsup_{\tau\to 0} \int_0^T\int_\Omega |\nabla\varrho^\tau|^2\leq \int_0^T\int_\Omega |\nabla\varrho|^2$, which can be obtained by the very same estimates (using the strong $L^2_tL^2_x$ convergence to handle the extra term involving $\nabla V$). This proof is presented in Section 4, after two preliminary sections, one on the properties of the solution in continuous-time (Section 2) and one on the properties of the JKO scheme (Section 3).

Then, a similar argument is proposed for the convergence of the second derivatives in space. The strategy consists in finding a first-order quantity which decreases along iterations of the JKO scheme and its dissipation is a second-order quantity which is, up to terms which tend to $0$ when $\tau\to 0$, the very same dissipation which could be obtained for the same quantity along the continuous-in-time flow of the PDE. This is done by looking at the evolution in time of the quantities
$$ F_p(\varrho)\colonequals \frac1p \int_{\Omega}\left|\frac{\nabla \varrho}{\varrho}+\nabla V\right|^p \,d\varrho.$$
The particular case $p=2$ is the most important one, as the functional $F_2$, sometimes called {\it Fischer information} (in particular in the case $V=0)$ naturally appears in the Fokker-Planck equation as the dissipation of the entropy $J$ along the solution of the equation (more precisely, we do have $\partial_t J(\varrho_t)=-2F_2(\varrho_t)).$ This fact is widely used in functional inequalities as for instance in the Bakry-Emery theory (see \cite{BGL}, for instance). Our analysis will be based on the evolution in time of the functionals $F_p$ along the JKO scheme, and on the evolution of $F_2$ on both the JKO and the continuous-in-time equation. It is possible to differentiate $F_2$ in time and obtain several terms including the main one $-\int_\Omega \varrho|D^2(\log\varrho+V)|^2$. The same computation may be done on the JKO scheme using the so-called five-gradients inequality introduced by the second author in \cite{five.grad} and applied to the Fokker-Planck equation in \cite{F.DiM}. This requires a finer analysis than what is done in \cite{F.DiM} since the remainders of the inequality will be crucial. Moreover, the dissipation along the steps of the JKO scheme does not provide exactly the desired term $-\int_\Omega \varrho^\tau|D^2(\log\varrho^\tau+V)|^2$ but includes an error term of the order of $||D^2\varphi^\tau_k||_{L^\infty}$, where $\varphi^\tau_k$ is the Kantorovich potential in the optimal transport from $\varrho^\tau_k$ to $\varrho^\tau_{k-1}$.

The strategy to get rid of this error term is the following: we prove uniform upper and lower bounds on $\varrho^\tau$ (which can be proven on the JKO scheme and are now well-known); once we couple these bounds with uniform $C^{0,\alpha}$ estimates on $\varrho^\tau$ this implies a uniform bound on the potentials $\varphi^\tau_k$ in $C^{2,\alpha}$. Uniform $C^{0,\alpha}$ bounds on $\varrho^\tau$ are obtained in a non-optimal way: we indeed suppose $F_p(\varrho_0)<+\infty$ for $p>d$ and prove that this quantity stays bounded in time, which implies the H\"older behavior because of standard Sobolev injections. We then use, again, the lower bound on the densities to obtain upper bounds of the form $|\nabla \varphi^\tau_k|\leq C\tau^\beta$ for $\beta>0$, which means that the gradient of the Kantorovich potentials tend uniformly to $0$ (this estimate is obtained using an argument from \cite{Bouchitte} which provides an $L^\infty$ estimate on the displacement $|T(x)-x|$) and this, together with the uniform H\"older bound on $D^2\varphi^\tau_k$, provides uniform convergence to $0$ of the Hessian as well. 

Once we get rid of the error terms, the fact that the estimates in discrete and in continuous time are essentially the same allows to obtain $\int_0^T \int_\Omega \varrho^\tau|D^2(\log\varrho^\tau+V)|^2\to \int_0^T \int_\Omega \varrho|D^2(\log\varrho+V)|^2$ which provides strong $L^2_{t,x}$ convergence of $\sqrt{\varrho^\tau}D^2\log\varrho^\tau$ and, after carefully using again the upper and lower bounds on $\varrho^\tau$, we obtain $D^2\varrho^\tau\to D^2\varrho$ in $L^2_{t,x}$.
\bigskip

\textbf{Acknowledgement.} The first author acknowledges the support of the ANR project MAGA (ANR-16-CE40-0014) and of the Lagrange Mathematics and Computation Research Center project on Optimal Transportation. The second author acknowledges the partial support of  the Austrian Science Fund (FWF) through the project F65 ``Taming Complexity in Partial Differential Systems'', as well as the support of \'Ecole normale sup\'erieure de Lyon and its scholarship program Amp\`ere for the year spent in Lyon when the work leading to this paper started as a part of his master thesis.

\section{Basics on the Fokker-Planck equation}

In our note we will consider the Fokker-Planck equation in a finite interval $[0,T]$  and a convex bounded domain  $\Omega \subset \mathbb{R}^d$ whose boundary $\partial \Omega$ is smooth enough. The drift in the equation will be of gradient type and autonomous.  We denote by $\vec{n}$   the  exterior unit normal vector of the boundary $\partial \Omega.$ We consider the Cauchy problem for the Fokker-Planck equation with no-flux boundary condition, i.e.,
\begin{equation}\label{F-P}
\begin{cases}
\partial_t \varrho(t, x)=\Delta \varrho(t,x)+\text{div} (\varrho(t,x) \nabla V(x)),  &\text{  } (t, x) \in (0,T]\times \Omega,\\
\nabla \varrho(t,x) \cdot\vec{n}(x)+\varrho(t,x) \nabla V(x)\cdot\vec{n}(x)=0, &\text{  } (t,x) \in [0,T]\times \partial \Omega,\\
\varrho(0,x)=\varrho_0(x), &\text{ } x\in \Omega,
\end{cases}
\end{equation}
where   $\varrho_0 \in \mathcal{P}(\Omega)\cap L^1_+(\Omega).$ \\
The initial total mass is $\int_{\Omega}\varrho_0(x)dx=1$ and  it can be formally seen  that it is  preserved.
Well-known results on  parabolic differential equations (see \cite{Lady},  \cite{A.L} ) let us to state the following properties of \eqref{F-P}:
\begin{proposition}\label{th:reg.F-P}
Let $[0,T]$ be a finite interval and $\Omega$ be bounded domain whose boundary $\partial \Omega$ is Lipschitz continuous. Then the following hold: 
 \begin{itemize}
 \item if $V$ is Lipschitz continuous and the initial data  $\varrho_0\in L^2(\Omega),$ then there exists a unique solution  $\varrho $ of (\ref{F-P})  in $ L^2([0,T];H^1(\Omega))\cap C([0,T];L^2(\Omega)).$  If $\varrho_0\in H^1(\Omega),$  then we have $\varrho\in L^2([0,T];H^2(\Omega))\cap C([0,T];H^1(\Omega))$.
 \item if $\varrho_0 \in C(\bar{\Omega})$ and  $V \in C^2(\bar{\Omega}),   $ then $\varrho\in C([0,T]\times \Omega)$ is differentiable with respect to $t$ in $(0,T] \times \bar{\Omega}$;  $\varrho(t,\cdot)$ belongs to $W^{2,p}(\Omega)$ for every $p\geq 1$. 
\item If $\varrho_0$ is bounded from below and above by two positive constants, then the same (for possibly different constants) holds for $\varrho$.
\item if $\varrho_0 \in C(\bar{\Omega}),$ $ \partial \Omega$ has $C^{2+\alpha}$ regularity and $V \in C^{2+\alpha}(\bar{\Omega})$ for $\alpha \in (0,1),$ then $\varrho \in {C}^{1+\frac{\alpha}{2}, 2+\alpha}((0,T] \times \bar{\Omega})$. If moreover $\partial \Omega$ and $V$ are $C^{3+\alpha}$, then $\varrho(t,\cdot)\in C^{3+\alpha}(\bar{\Omega})$ for every $t>0$.
\end{itemize}
\end{proposition}

 
 Once we know that the solution of (\ref{F-P}) exists, is unique, and is smooth, we are interested in evaluating, and in particular differentiating in time, some quantities involving the solution. First we consider the following classical statement
 \begin{proposition}\label{continuous L2 derivative}
 Let $\varrho_0\in L^2(\Omega)$ and  $\varrho_t$ be the unique solution of (\ref{F-P}). Then we have
  $$
 \displaystyle \int_{\Omega}\varrho_T^2(x)\,dx-\int_{\Omega}\varrho_0^2(x)\,dx=-2\int_0^T \int_{\Omega}|\nabla \varrho_t(x)|^2\,dxdt -2\int_0^T \int_{\Omega} \varrho_t(x) \, \nabla \varrho_t(x) \cdot \nabla V(x)\,dxdt. 
 $$
\end{proposition}
\begin{proof}
In order to obtain this result it is enough to differentiate in time the function $t\mapsto \int_\Omega\varrho_t^2$. To do this, we use a Lemma \ref{le:partial_t rho}, which is a general lemma from functional analysis. Since $\varrho_t \in L^2([0,T];H^1(\Omega)) $ we have $\partial_t \varrho_t=\Delta \varrho_t +\text{div}(\varrho_t \nabla V) \in L^2([0,T];H^{-1}(\Omega)). $ Thus, $\varrho_t$ satisfies  Lemma \ref{le:partial_t rho}  in the case of $V=H^1(\Omega)$ and $H=L^2(\Omega).$  Consequently, we have that $\varrho_t \in C([0,T];L^2(\Omega))$ and 
 $$
 \displaystyle \int_{\Omega}\varrho_T^2(x)\,dx-\int_{\Omega}\varrho_0^2(x)\,dx=2\int_0^T\langle \partial_t \varrho_t,\varrho_t \rangle_{H^{-1},H^1}\,dt
$$
and it is enough to use the expression for $\partial_t\varrho$ and integrate in time to obtain the result.
\end{proof}
 
 In the above proof we mentioned a general functional analysis fact, which is recalled here below. To introduce it, let us consider a Hilbert space $H$ endowed with the norm $||\cdot||_{H}$, a  Banach space $V$, and we assume that $V$ is reflexive, $V\subset H$ with dense and bounded embedding. The following Lemma is proven, for instance in \cite[Lemma 1.2, page 260]{Temam}.
  \begin{lemma}\label{le:partial_t rho}
  The following inclusion 
  $$
  L^2([0,T];V)\cap H^1([0,T];V')\subset C([0,T];H)
  $$
  holds true. Moreover, for any $g \in L^2([0,T];V)\cap H^1([0,T];V')$ there holds 
  $$
  t\rightarrow ||g(t)||^2_{H} \in W^{1,1}(0,T)
  $$
  and
  \begin{equation}
  \dfrac{d}{dt}||g(t)||^2_{H}=2\langle g'(t),g(t) \rangle_{V',V}  \text{   } a.e. \text{    on   }   (0,T).
  \end{equation}
  \end{lemma}
 
In the above computation, we saw that a zero-order quantity (here $\int \varrho^2$) is the integral in time of a first-order quantity (which is in our case given by $\int|\nabla \varrho^2|+ \varrho\nabla \varrho \cdot \nabla V$). We now need to look at higher-order quantities. In particular, we will consider the functional
\begin{equation}\label{eq:Fisher}
 F_p(\varrho)\colonequals \frac1p \int_{\Omega}\left|\frac{\nabla \varrho}{\varrho}+\nabla V\right|^p \,d\varrho
\end{equation}
defined for $\varrho \in W^{1,1}(\Omega)\cap \mathcal{P}(\Omega).$ We will mainly look at $F_2$.

\begin{lemma}\label{lem:Fisher F-P} 
Suppose $0<T< +\infty,$  $\Omega$ is a bounded domain with $C^{3+\alpha}$ boundary,  $V \in C^{3+\alpha}(\bar{\Omega})   $ for some $\alpha \in (0,1)$ and $\varrho_0 \in C(\bar{\Omega})$ is positive. If  $\varrho$ is  the solution of (\ref{F-P}), then for $t>0$ we have 
\begin{equation*}
\partial_t F_2(\varrho_t)=-\int_{\Omega} |D^2(\log{\varrho}+V)|^2\varrho \,dx-\int_{\Omega}(\nabla(\log{\varrho}+V) )^{T}\cdot D^2V\cdot \nabla(\log{\varrho}+V)  \varrho\, dx+
\end{equation*}
\begin{equation}\label{eq(4)}\int_{\partial \Omega} (\nabla(\log{\varrho}+V) )^{T}\cdot D^2(\log{\varrho}+V) \cdot \vec{n}\varrho\, d\mathcal{H}^{d-1}.
\end{equation}
\end{lemma}
\begin{proof} The assumptions provide that  $\varrho(t,\cdot)$ is also positive  and belongs to $ C^{3+\alpha}(\bar{\Omega})$ for every $t\in (0,T].$ We have
$$\small
 2\displaystyle \partial_t F_2(\varrho_t) =\partial_t \int_{\Omega} \sum_{i=1}^{d}\left(\frac{\partial_i \varrho}{\varrho}+\partial_i V \right)^2 \varrho\, dx =
 $$
$$ \small
 \sum_{i=1}^d \int_{\Omega} \partial_t \varrho \left(\frac{\partial_i\varrho}{\varrho}+\partial_i V\right)^2 dx + 2\sum_{i=1}^d \int_{\Omega}
\displaystyle  \left(\frac{\partial_i\varrho}{\varrho}+\partial_i V\right) \left(\partial_t \partial_i\varrho-\frac{\partial_i \varrho \, \partial_t \varrho }{\varrho}\right)dx=
$$
$$\small
\displaystyle \int_{\Omega}\sum_{i=1}^d \partial_t\varrho \left[ \left(\partial_i V\right)^2- \left(\frac{\partial_i \varrho}{\varrho}\right)^2 \right]dx+ 2\sum_{i=1}^d \int_{\Omega}  \left(\frac{\partial_i\varrho}{\varrho}+\partial_i V\right) \partial_t \partial_i \varrho\, dx.
$$
We look at the different parts of the last integral. First we use the equation $ \partial_t \varrho=\Delta \varrho+\text{div} (\varrho \nabla V)$ and integrate  by parts.
$$\small
\displaystyle \int_{\Omega}\sum_{i=1}^d \partial_t\varrho \left[ \left(\partial_i V\right)^2- \left(\frac{\partial_i \varrho}{\varrho}\right)^2 \right]\,dx= \int_{\Omega}\sum_{i,j=1}^d \left[ \left(\partial_i V\right)^2- \left(\frac{\partial_i \varrho}{\varrho}\right)^2 \right] \partial_j \left(\partial_j \varrho +\varrho\, \partial_j V\right)\,dx= 
$$
$$
\displaystyle -2 \sum_{i,j=1}^d  \int_{\Omega}\left[ \varrho\, \partial_i V \partial_{ij}V-\frac{\partial_i \varrho \, \partial_{ij}\varrho}{\varrho}+\frac{(\partial_i \varrho)^2 \, \partial_{j}\varrho}{\varrho^2}\right] \left(\frac{\partial_j}{\varrho}  + \partial_j V\right)\,dx+
$$
$$
\displaystyle \sum_{i,j=1}^d \int_{\partial \Omega} \left[ \left(\partial_i V\right)^2- \left(\frac{\partial_i \varrho}{\varrho}\right)^2 \right] \left(\frac{\partial_j \varrho}{\varrho} 
  + \partial_j V\right) \varrho n_j\, d\mathcal{H}^{d-1}.
$$

We now compute the second part
$$
\displaystyle 2\sum_{i=1}^d \int_{\Omega} \left(\frac{\partial_i\varrho}{\varrho}+\partial_i V\right) \partial_t \partial_i \varrho \, dx=2\sum_{i,j=1}^d \int_{\Omega} \left(\frac{\partial_i\varrho}{\varrho}+\partial_i V\right)(\partial_{ijj}\varrho+\partial_{ij}(\varrho \partial_j V))\,dx=
$$
$$\small \displaystyle
-2\sum_{i,j=1}^d \int_{\Omega} \left( \frac{\partial_{ij}\varrho}{\varrho} +\partial_{ij} V-\frac{\partial_i \varrho \, \partial_j \varrho}{{\varrho}^2}\right)(\partial_{ij}\varrho+\varrho \partial_{ij}V+ \partial_i \varrho\, \partial_j V)\,dx+
$$
$$\small \displaystyle
2\sum_{i,j=1}^{d} \int_{\partial \Omega}\left(\frac{\partial_i\varrho}{\varrho}+\partial_i V\right)(\partial_{ij}\varrho+\varrho \partial_{ij}V+ \partial_i \varrho\,  \partial_j V)n_j \,d\mathcal{H}^{d-1}=
$$
$$\small \displaystyle
-2\sum_{i,j=1}^d \int_{\Omega} \left( \frac{\partial_{ij}\varrho}{\varrho} +\partial_{ij} V-\frac{\partial_i\varrho \,  \partial_j \varrho}{{\varrho}^2}\right)^2 \varrho\, dx
$$ 
$$\displaystyle
-2\sum_{i,j=1}^d \int_{\Omega} \left( \frac{\partial_{ij}\varrho}{\varrho} +\partial_{ij} V-\frac{\partial_i \varrho \, \partial_j \varrho}{{\varrho}^2}\right)\partial_i \varrho \left(\frac{\partial_j\varrho}{\varrho}+\partial_j V\right)\,dx
$$
$$\small \displaystyle
+2\sum_{i,j=1}^{d} \int_{\partial \Omega}\left(\frac{\partial_i\varrho}{\varrho}+\partial_i V\right)(\partial_{ij}\varrho+\varrho \partial_{ij}V+ \partial_i \varrho\, \partial_j V)n_j\, d\mathcal{H}^{d-1}.
$$
We consider the integrals over the boundary of $\Omega.$ Because of the no-flux boundary condition the first boundary integral vanishes:
$$
\displaystyle \sum_{i,j=1}^d \int_{\partial \Omega} \left[ \left(\partial_i V\right)^2- \left(\frac{\partial_i \varrho}{\varrho}\right)^2 \right] \left(\frac{\partial_j \varrho}{\varrho} 
  + \partial_j V\right) \varrho n_j \,d\mathcal{H}^{d-1}=0.
$$
Because of the same reason the second boundary integral can be written as follows
$$\small \displaystyle
2\sum_{i,j=1}^{d} \int_{\partial \Omega}\left(\frac{\partial_i\varrho}{\varrho}+\partial_i V\right)(\partial_{ij}\varrho+\varrho \partial_{ij}V+ \partial_i \varrho\cdot  \partial_j V)n_j \,d\mathcal{H}^{d-1}=
$$
$$\small \displaystyle
2\sum_{i,j=1}^{d} \int_{\partial \Omega} \varrho \left(\frac{\partial_i\varrho}{\varrho}+\partial_i V\right)\left(\frac{\partial_{ij}\varrho}{\varrho}+ \partial_{ij}V-\frac{\partial_i \varrho\cdot  \partial_j \varrho}{{\varrho}^2}\right)n_j \,d\mathcal{H}^{d-1}.
$$

Consequently, we have 
$$\small
 \displaystyle \partial_t F_2(\varrho)=-\sum_{i,j=1}^d \int_{\Omega} \left( \frac{\partial_{ij}\varrho}{\varrho} +\partial_{ij} V-\frac{\partial_i \varrho \, \partial_j \varrho}{{\varrho}^2}\right)^2 \varrho\, dx-
$$
$$ 
 \sum_{i,j=1}\int_{\Omega} \left(\frac{\partial_i\varrho}{\varrho}+\partial_i V\right)\partial_{ij}V \left(\frac{\partial_j\varrho}{\varrho}+\partial_j V\right) \varrho \,dx
$$ 
$$
 +\sum_{i,j=1}^{d} \int_{\partial \Omega} \varrho \left(\frac{\partial_i\varrho}{\varrho}+\partial_i V\right)\left(\frac{\partial_{ij}\varrho}{\varrho}+ \partial_{ij}V-\frac{\partial_i \varrho\,  \partial_j \varrho}{{\varrho}^2}\right)n_j \,d\mathcal{H}^{d-1}.
 $$
 If we take into account that we have
 $$\partial_{i}(\log{\varrho}+V)=\left(\frac{\partial_i\varrho}{\varrho}+\partial_i V\right) \text{   and  } \partial_{i,j}(\log{\varrho}+V)=\frac{\partial_{ij}\varrho}{\varrho} +\partial_{ij} V-\frac{\partial_i \varrho \, \partial_j \varrho}{{\varrho}^2},$$ we get the desired equality.
\end{proof}

The last term in formula \eqref{eq(4)} can be re-written using the following lemma.
\begin{lemma}\label{vD2hn}
Suppose $\Omega=\{h<0\}$ for a smooth function $h:\R^d\to\R$ with ${\nabla }h\neq 0$ on $\{h=0\}$, so that the exterior normal vector at $x\in\partial\Omega$ is given by $\vec{n}(x)=\nabla h(x)/|\nabla h(x)|$. Let $v:\Omega\to\R^d$ be a smooth vector field such that $v\cdot \vec{n}=0$ on $\partial \Omega$. Then we have the following equality for every $x\in \partial\Omega$
$$v(x)^T \cdot Dv(x)\cdot n(x) =-\frac{v(x)^T\cdot D^2h(x) \cdot v(x) }{|\nabla h(x)|}.$$
\end{lemma}
\begin{proof}
Given $x\in\partial\Omega$, we consider a smooth curve $\gamma:(-t_0,t_0)\to \partial\Omega$ with $\gamma(0)=x$ and write the equality $v(\gamma(t))\cdot \nabla h(\gamma(t))=0$ for every $t$. Differentiating w.r.t. $t$ we obtain
$$\gamma'(t)^T\cdot Dv(\gamma(t))\cdot \nabla h(\gamma(t))+v(\gamma(t))^T\cdot D^2 h(\gamma(t))\cdot \gamma'(t)=0.$$
We can take $t=0$ and choose a curve with $\gamma'(0)=v(x)$ since $v$ is tangent to the surface $\partial\Omega$, thus obtaining
$$v(x)^T \cdot Dv(x)\cdot \nabla h(x) =-v(x)^T\cdot D^2h(x) \cdot v(x) .$$
It is then enough to divide by $|\nabla h(x)|$ in order to get the claim.
\end{proof}
We then obtain the following formula.
\begin{corollary}\label{coroF2rho}
Suppose $0<T< +\infty,$ take $\Omega=\{h<0\}$ a bounded domain defined as the negativity set of a function $h\in C^2$ with  $\nabla h\neq 0$ on $\{h=0\}$, and $V\in C^2(\bar{\Omega}).$ Given a strictly positive $\varrho_0 \in H^1({\Omega})$ initial datum, let $\varrho$ be  the solution of (\ref{F-P}). We then have
\begin{eqnarray*}
F_2(\varrho_T)-F_2(\varrho_0)&=&-\int_0^Tdt\int_{\Omega} |D^2(\log{\varrho}+V)|^2\varrho \,dx\\
&&-\int_0^Tdt\int_{\Omega}(\nabla(\log{\varrho}+V) )^{T}\cdot D^2V\cdot \nabla(\log{\varrho}+V)  \varrho\, dx\\
&&-\int_0^Tdt\int_{\partial \Omega} (\nabla(\log{\varrho}+V) )^{T}\cdot D^2h \cdot  (\nabla(\log{\varrho}+V) )\varrho\, d\mathcal{H}^{d-1}.
\end{eqnarray*}
\end{corollary}
\begin{proof}
The result is obtained by first generalizing formula \eqref{eq(4)} to the case of $C^2$ regularity by approximation, then integrating in time, using the continuity in $H^1$ of the solution at $t=0$, and finally re-writing the boundary term using Lemma \ref{vD2hn}
\end{proof}

\section{The JKO scheme for the Fokker-Planck equation}

We define the functional $J\colon \mathcal{P}(\Omega)\to \mathbb{R}$ as follows
\begin{equation}\label{Entropy func}
\displaystyle J(\mu)= \begin{cases}\displaystyle
 \int_{\Omega}\dfrac{d\mu}{dx} \log{ \left(\dfrac{d\mu}{dx}\right)}dx+\int_{\Omega}V d\mu &\text{ if } \mu \ll dx\\
+\infty &\text{otherwise.}
\end{cases}
\end{equation}

The main achievement of \cite{JKO} is to view the Fokker-Planck equation (\ref{F-P}) as a gradient flow of the functional $J$ in the metric space $(\mathcal{P}(\Omega),{W}_2)$ (see also \cite{AGS, OTAM,Villani}) and to define a discrete iterated scheme converging to the solution. This scheme is nowadays called Jordan-Kinderlehrer-Otto scheme.

Given $N\in \mathbb{N},$ we set  $\tau \colonequals \frac{T}{N}$ and  we assume  that  the initial data $\varrho_0$  satisfies   $J(\varrho_0)< \infty.$ We define recursively a sequence of probability measures $\{\varrho_k^{\tau}\}_{k=0}^{N}$  such that $\varrho_0^{\tau}\colonequals \varrho_0$ and 

 \begin{equation}\label{Sequence}
 \displaystyle \varrho_{k+1}^{\tau} \in \text{argmin}_{\varrho}\left\{ J(\varrho)+\dfrac{{W}^2_2(\varrho,\varrho_k^{\tau})}{2\tau}\right\}
\end{equation}
We use this sequence to build a curve $\varrho^{\tau}(t)$ in the space of probability measures defined via $\varrho^{\tau}(0)=\varrho_0$ and
$$\varrho^{\tau}(t)=\varrho_{k+1}^{\tau}  \text{   for  } t\in (k\tau,(k+1)\tau] \text{  and } k \in \mathbb{N}\cup\{0\}.$$
In the original work by R. Jordan, D. Kinderlehrer and F. Otto (\cite{JKO}) the above scheme was considered for $\Omega=\mathbb{R}^d,$ and the following important theorem was proven.
\begin{theorem}\label{th:JKO original}
Let $V \in C^{\infty}(\mathbb{R}^d),$  $\varrho_0\in \mathcal{P}_2(\mathbb{R}^d)$ satisfies $J(\varrho_0)<\infty,$ and for given $\tau>0,$ let $\{\varrho_k^{\tau}\}_{k\in \mathbb{N}}$ be defined recursively by (\ref{Sequence}). Define the interpolation 
$$\varrho^{\tau}(t)=\varrho_{k+1}^{\tau}  \text{   for  } t\in (k\tau,(k+1)\tau] \text{  and } k \in \mathbb{N}\cup\{0\}.$$
Then as $\tau \to 0,$ $\varrho^{\tau}(t)\to \varrho(t)$  weakly in $L^1(\mathbb{R}^d) $ for all $t \in (0,\infty),$ where $\varrho \in C^{\infty}((0,\infty)\times \mathbb{R}^d)$ is the unique solution of $$\partial_t \varrho=\Delta \varrho+\mathrm{div}(\varrho \nabla V)$$
with initial condition $\varrho(t)\to \varrho_0$ strongly in $L^1(\mathbb{R}^d)$ for $t \to 0.$  Moreover, $\varrho^{\tau}\to \varrho$  strongly in $L^1((0,T)\times \mathbb{R}^d)$  for all $T<\infty.$
 \end{theorem}
 
In this paper we are instead interested in the case where $\Omega$ is an open bonded subset of $\mathbb{R}^d.$ The details of the JKO scheme for bounded domains are, for instance, given in \cite[Chapter 8]{OTAM}. In this case it is important to emphasize that the limit curve of  the JKO scheme not only  solves the Fokker-Planck equation but also satisfies the \textit{no-flux} boundary condition. More precisely, we summarize here the properties we need about the JKO scheme for bounded  $\Omega:$
\begin{theorem}\label{th:Properties}
Let  $[0,T]$ be a finite interval and  $\Omega$  be a bounded domain of $\mathbb{R}^d$ whose boundary $\partial \Omega$ is Lipschitz continuous. We assume $V$ is  Lipschitz continuous and   the initial data $\varrho_0 \in \mathcal{P}(\Omega)\cap L^1_+(\Omega)$  satisfies   $J(\varrho_0)< \infty.$ Then the following hold:
\begin{enumerate}
\item The functional $J$ has a unique minimum over $\mathcal{P}(\Omega).$ In particular $J$ is bounded from below. Moreover, for each $\tau>0,$ the sequence $\{\varrho_k^{\tau}\}_{k=0}^{N}$ defined by the formula (\ref{Sequence}) is well-defined (there is a unique minimizer at every step).
\item For any $k \in \{1,...,N\},$ the optimizer $\varrho_{k}^{\tau}$ is strictly positive, Lipschitz continuous, and satisfies
\begin{equation}\label{Optimality condition}
\log{\varrho_{k+1}^{\tau}}+V+\dfrac{\varphi_k}{\tau}=constant 
\end{equation}
where $\varphi_k$ is the Kantorovich potential from $\varrho_{k+1}^{\tau}$ to $\varrho_{k
}^{\tau}.$ 
\item For every $\tau>0,$ the sequence $\{\varrho_k^{\tau}\}_{k=0}^{N}$ satisfies
\begin{equation}\label{W_2 bound}
\displaystyle \sum_{k=0}^{N-1}\dfrac{{W}^2_2(\varrho_k^{\tau},\varrho_{k+1}^{\tau})}{\tau}\leq 2(J(\varrho_0)-\inf J).
\end{equation}
\item There exists a $\frac{1}{2}$-H\"older and absolutely  continuous curve $\varrho_t$ in $(\mathcal{P}(\Omega),{W}_2)$ such that ${W}_2(\varrho^{\tau}_{t}, \varrho_t)\rightarrow 0$  uniformly as $\tau\rightarrow 0.$ Moreover,  $\varrho_t  $ satisfies the Fokker-Planck equation (\ref{F-P}) in the distributional sense.
\end{enumerate}
 \end{theorem}
Some estimates in $L^p $ have been established in  \cite{F.DiM} when the domain $\Omega$ is convex.
 \begin{theorem}\label{th:L^p estimates}
 Let  $[0,T]$ be a finite interval and  $\Omega$  be a convex bounded domain of $\mathbb{R}^d$ whose boundary $\partial \Omega$ is Lipschitz continuous. We assume $V$ is  Lipschitz continuous and    the initial data $\varrho_0 \in \mathcal{P}(\Omega)\cap L^1_+(\Omega)$  satisfies   $J(\varrho_0)< \infty.$
 
 (i) Given $f:\R_+\to\R$ satisfying the $d$-McCann condition (i.e. $[0,\infty) \ni s \mapsto f(s^{-d})s^d$ is convex and decreasing), then we have 
 \begin{equation}\label{Flow interchange}
 \displaystyle \int_{\Omega} f(\varrho_k^{\tau})\,dx \geq \int_{\Omega}f(\varrho_{k+1}^{\tau})\,dx+\tau \int_{\Omega}\left(f''(\varrho_{k+1}^{\tau})|\nabla \varrho_{k+1}^{\tau}|^2+\varrho_{k+1}^{\tau} f''(\varrho_{k+1}^{\tau})\nabla \varrho_{k+1}^{\tau} \cdot \nabla V\right)\,dx.
 \end{equation}

 (ii) Suppose $\varrho_0 \in L^p(\Omega),$ with $p<\infty.$  Then, for any $k\in \{1,...,N\},$ we have 
 \begin{equation}\label{L^p estimate}
 \displaystyle \int_{\Omega} (\varrho_{k}^{\tau})^p dx \geq \left(1-\tau \frac{p(p-1)}{4} \text{Lip}(V)^2 \right)\int_{\Omega}(\varrho_{k+1}^{\tau})^p dx.
 \end{equation}
 
(iii) In particular, if $\tau$ is  small enough (depending on $V$ and $p$) and under the assumptions of $(i)$ and $(ii),$  the norm $||\varrho_{t}^{\tau}||_{L^p}$  grows at most exponentially in time for $p \in [1, \infty]$

\end{theorem}

\section{Weak and strong convergence of the JKO scheme}\label{chap2}
In this section we use the bounds provided by  Theorem \ref{th:L^p estimates} on the solutions of the JKO scheme to improve its convergence to the solution of the Fokker-Planck equation up to strong convergence in $L^2([0,T]; H^1(\Omega)).$ As a starting point, we first consider weak convergence in the same space.
\begin{proposition}\label{p^t bounds}
Under the same assumptions of Theorem \ref{th:L^p estimates} and $\varrho_0 \in L^2(\Omega)$, the curve $\varrho^{\tau}_t$ is uniformly bounded with respect to $\tau$ in $L^{\infty}([0,T];L^2(\Omega))$ and $L^2([0,T]; H^1(\Omega)).$ Moreover, if $\varrho_t$ is the solution of (\ref{F-P}), then  $\varrho^{\tau}_t \rightharpoonup \varrho_t $  in $L^2([0,T]; H^1(\Omega)).$
\end{proposition}
\begin{proof}
We use in the inequality (\ref{Flow interchange}) in Theorem \ref{th:L^p estimates} for the function $f(s)=s^2$. Then, for each $k\in \{0,,...,N-1\},$ $\varrho_{k+1}^{\tau}\in L^2(\Omega)   $ and  we have
\begin{equation}\label{eq:flow interchange}
\displaystyle \int_{\Omega} (\varrho_k^{\tau})^2\,dx \geq \int_{\Omega}(\varrho_{k+1}^{\tau})^2\,dx+2\tau \int_{\Omega}\left(|\nabla \varrho_{k+1}^{\tau}|^2+\varrho_{k+1}^{\tau} \nabla \varrho_{k+1}^{\tau} \cdot \nabla V\right)\,dx.
\end{equation}
By the Young's inequality we have
$$
\displaystyle  \int_{\Omega} \varrho_{k+1}^{\tau} \nabla \varrho_{k+1}^{\tau} \cdot \nabla V\,dx\geq -\frac 12 \int_{\Omega} |\nabla \varrho_{k+1}^{\tau}|^2-\frac12 \int_{\Omega} ( \varrho_{k+1}^{\tau})^2|\nabla V|^2\,dx $$
The estimate above implies 
$$
\displaystyle \int_{\Omega} (\varrho_{k}^{\tau})^2dx -\int_{\Omega} (\varrho_{k+1}^{\tau})^2dx+ \tau \,\text{Lip}(V)^2 \int_{\Omega}(\varrho_{k+1}^{\tau})^2dx \geq \tau \int_{\Omega}|\nabla \varrho_{k+1}^{\tau}|^2dx.$$
We sum the inequalities above with respect to $k,$ then 
\begin{equation}\label{int-est}
\displaystyle \int_{\Omega} (\varrho_{0}^{\tau})^2dx -\int_{\Omega} (\varrho_{N}^{\tau})^2dx+ \text{Lip}(V)^2 \sum_{k=0}^{N-1} \tau \int_{\Omega}(\varrho_{k+1}^{\tau})^2dx \geq \sum_{k=0}^{N-1}\tau \int_{\Omega}|\nabla \varrho_{k+1}^{\tau}|^2dx.
\end{equation}
By the definition of the curve $\varrho^{\tau}_t, $ we have
$$
\displaystyle ||\varrho^{\tau}||^2_{L^2([0,T];L^2(\Omega))}=\sum_{k=0}^{N-1} \tau \int_{\Omega}(\varrho_{k+1}^{\tau})^2dx
$$ 
and
  $$
  \displaystyle ||\varrho^{\tau}||_{L^2([0,T];H^1(\Omega))}^2=\sum_{k=0}^{N-1} \left(\tau \int_{\Omega}(\varrho_{k+1}^{\tau})^2\,dx+\tau \int_{\Omega}|\nabla \varrho_{k+1}^{\tau}|^2dx \right).
  $$
  The inequality (\ref{L^p estimate}) in Theorem \ref{th:L^p estimates} (which is actually proven exactly as in the computations above, choosing better coefficients in the Young inequality so that the $H^1$ part disappears), provides uniform bounds on $||\varrho^{\tau}_t||_{L^2(\Omega)}$, which guarantees the bound in $L^{\infty}([0,T];L^2(\Omega))$ and, using \eqref{int-est}, also in $L^2([0,T]; H^1(\Omega)).$ 
  
  Since $L^2([0,T]; H^1(\Omega))$ is reflexive, when $\tau\to 0$ it is easy to find a weak limit $\rho \in L^2([0,T]; H^1(\Omega))$ and this limit necessarily coincides with the limit in the Wasserstein sense, i.e. the unique solution of (\ref{F-P}), which proves the last part of the statement.
  \end{proof} 
  
  We now want to start proving strong convergences of $\varrho^\tau$ to $\varrho$. A first step will make use of the well-known Aubin-Lions lemma for time-dependent functions valued into functional spaces, but we need to handle the time derivative. As the functions $t\mapsto \varrho^\tau$ are discontinuous, for simplicity (instead of evoking modified versions of the Aubin-Lions compactness criterion using functions which are BV in time), we  define a new family of interpolations which help in obtaining the desired result. 
  
Given $\varepsilon \in (0,1),$ we consider  another curve $\varrho^{\tau,\varepsilon}_{t}$ such that $\varrho^{\tau,\varepsilon}_{0}=\varrho_0$ and for $t\in (0,T]$
$$
\varrho^{\tau,\varepsilon}_{t}=\begin{cases} \varrho^{\tau}_{k+1} &\text{ if } t\in (k\tau,k\tau+(1-\varepsilon)\tau] \text{ and } k \in \{0,...,N-2\} \\
\varrho^{\tau}_{k+1}\frac{(k+1)\tau-t}{\varepsilon \tau}+\varrho^{\tau}_{k+2}\frac{t-(k+1)\tau+\varepsilon \tau}{\varepsilon \tau} &\text{ if } t\in (k\tau+(1-\varepsilon)\tau,(k+1)\tau] \text{ and } k \in \{0,...,N-2\}\\
\varrho_N^{\tau} & \text{ if } t\in ((N-1)\tau, N \tau] 
\end{cases}.
$$
This curve is Lipschitz continuous in space and time and it time-derivative equals
 \begin{equation}\label{distributional time deriv}
 \displaystyle \partial_{t}\varrho^{\tau, \varepsilon}_{t}=\sum_{k=0}^{N-2}\dfrac{\varrho^{\tau}_{k+2}-\varrho^{\tau}_{k+1}}{\varepsilon \tau} \cdot \mathbf{1}_{(k\tau+(1-\varepsilon)\tau, (k+1)\tau)} \text{     a.e.    } t\in [0,T].
 \end{equation}
 Let $\mathcal{X}:=\{\psi \in \text{Lip}(\Omega): \int_\Omega\psi=0 \}$ and $||\cdot||_{\text{Lip}}$ be the Lipschitz semi-norm, which is actually a norm on this space where the average is prescribed, then it is easy to check that $(\mathcal{X},||\cdot||_{\text{Lip}})$ forms a Banach space. We denote by $\mathcal{X}'$ the dual space of $\mathcal{X}$. This space includes all measures on $\Omega$ and it is well-known that we have $W_1(\mu,\nu)=||\mu-\nu||_{\mathcal X'}$ for every $\mu,\nu\in \mathcal P(\Omega)$.
 \begin{proposition}\label{p^tau,eps}
Under the same assumptions of Theorem \ref{th:L^p estimates} and $\varrho_0 \in L^2(\Omega)$, the following facts hold:
\begin{enumerate}

\item the curve $\varrho^{\tau,\varepsilon}_{t}$ converges to the solution of the Fokker-Planck equation uniformly in ${W}_2$ distance when $\tau\to 0$ (i.e. for every sequence $\varepsilon_j\in (0,1) $ and $\tau_j\to 0$ we do have this convergence).

\item  the curve $\varrho^{\tau,\varepsilon}_{t}$ is uniformly bounded in $L^{\infty}([0,T];L^2(\Omega))$ and $L^{2}([0,T];H^1(\Omega))$ with respect to $\varepsilon$  and $\tau.$

\item  we have $\displaystyle ||\varrho^{\tau} -\varrho^{\tau, \varepsilon}||_{L^{2}([0,T]; H^1(\Omega))}\leq C\sqrt{\varepsilon}$ for a constant $C$ independent of $\tau$ and $\varepsilon$.

\item $\partial_{t}\varrho^{\tau,\varepsilon}_{t}$ is uniformly bounded in $L^1([0,T];\mathcal{X}')$ with respect to $\varepsilon$  and $\tau.$
\end{enumerate}
\end{proposition}
\begin{proof}
Using the convexity of $\mu\mapsto W_2^2(\mu,\nu)$ it is easy to check that we have
\begin{equation}
{W}_2(\varrho^{\tau}_{t}, \varrho^{\tau,\varepsilon}_{t})\leq \sup_k W_2(\varrho^\tau_k,\varrho^\tau_{k+1})\leq \sqrt{\tau} \cdot \sqrt{2(J(\varrho_0)-\inf J)}.
\end{equation}
This shows that, as $\tau\rightarrow 0,$ the curves  $\varrho^{\tau}_{t}$ and $\varrho^{\tau,\varepsilon}_{t}$ tend to the same limit (which is  the solution of the Fokker-Planck equation), independently of $\varepsilon$.\\
 Since the curve $\varrho^{\tau}_t$ is uniformly bounded with respect to $\tau$ in $L^{\infty}([0,T];L^2(\Omega))$ by $C(T) ||\varrho_{0}||_{L^2(\Omega)}$, and $\varrho^{\tau,\varepsilon}$ is obtained by convex combinations of values of $\varrho^\tau_t$, the same bound is also true for $\varrho^{\tau,\varepsilon}$. 
 Moreover, we have
\begin{equation}\label{L^2 dist p^tau and p^tauep}
 ||\varrho^{\tau} -\varrho^{\tau, \varepsilon}||_{L^{2}([0,T]; L^2(\Omega))}=\sqrt{\frac{\varepsilon}{3}\sum_{k=0}^{N-2}\tau \int_{\Omega}(\varrho^{\tau}_{k+2}-\varrho^{\tau}_{k+1})^2 dx} \leq C(T)\sqrt{\varepsilon}\cdot ||\varrho_{0}||_{L^2(\Omega)}.
\end{equation}
Similar computations show
\begin{eqnarray*}
 ||\varrho^{\tau} -\varrho^{\tau, \varepsilon}||_{L^{2}([0,T]; H^1(\Omega))}&=&\sqrt{\frac{\varepsilon}{3}\sum_{k=0}^{N-2}\tau ||\varrho^{\tau}_{k+2}-\varrho^{\tau}_{k+1}||^2_{H^1}}\\
 & \leq& C\sqrt{\varepsilon}\sqrt{\sum_{k=0}^{N-2}\tau||\varrho^{\tau}_{k+2}||^2_{H^1}+\tau||\varrho^{\tau}_{k+1}||^2_{H^1}}\leq C\sqrt{\varepsilon}||\varrho^{\tau}||
 _{L^2([0,T];H^1(\Omega))}.
 \end{eqnarray*}
By Proposition \ref{p^t bounds} and the last two estimates we have $\displaystyle ||\varrho^{\tau} -\varrho^{\tau, \varepsilon}||_{L^{2}([0,T]; H^1(\Omega))}\rightarrow 0$  as $\varepsilon \rightarrow 0,$ uniformly in $\tau$.
We now look at the time derivative and we use the inequality  ${W}_1\leq W_2$ together with the identification of the distance $W_1$ as the dual norm of $\mathcal X$. We then have  \begin{equation}
\displaystyle \sum_{k=0}^{N-1} \frac{1}{\tau} ||\varrho_k^{\tau}-\varrho_{k+1}^{\tau}||^2_{\mathcal{X}'}\leq 2(J(\varrho_0)-\inf J).
\end{equation}
By the Cauchy–Schwarz inequality
\begin{equation}\label{Lip bound}
\displaystyle \sum_{k=0}^{N-1}  ||\varrho_k^{\tau}-\varrho_{k+1}^{\tau}||_{\mathcal{X}'}\leq \sqrt{N \displaystyle \sum_{k=0}^{N-1} ||\varrho_k^{\tau}-\varrho_{k+1}^{\tau}||^2_{\mathcal{X}'}}=\sqrt{ \frac{T}{\tau}\displaystyle \sum_{k=0}^{N-1}  ||\varrho_k^{\tau}-\varrho_{k+1}^{\tau}||^2_{\mathcal{X}'}}\leq \sqrt{2T (J(\varrho_0)-\inf J)}.
\end{equation}
The estimate (\ref{Lip bound}) and the expression (\ref{distributional time deriv}) we obtain
\begin{equation}
||\partial_{t}\varrho^{\tau,\varepsilon}||_{L^1([0,T];\mathcal{X}')}= \sum_{k=0}^{N-2} \varepsilon \tau||(\varrho^{\tau}_{k+2}-\varrho^{\tau}_{k+1})/\varepsilon \tau||_{\mathcal{X}'}\leq \sum_{k=0}^{N-1}  ||\varrho_k^{\tau}-\varrho_{k+1}^{\tau}||_{\mathcal{X}'}\leq  \sqrt{2T (J(\varrho_0)-\inf J)}. 
\end{equation}
The estimate above shows that $\partial_{t}\varrho^{\tau,\varepsilon}_{t}$ is uniformly bounded in $L^1([0,T];\mathcal{X}').$
\end{proof}
We  state here a version of the Aubin-Lions-Simon Compactness Theorem  which gives a compactness criterion in $L^p([0,T];B)$ for a Banach space $B.$ Its proof can be found, for instance, in  \cite[Corollary 6, page 87]{Simon}.
\begin{theorem}\label{Simon}
Let $X, B$ and $Y$ be Banach spaces such that $X\subset B \subset Y$  with compact embedding $X\hookrightarrow B.$ Let $E$ be a bounded set in $L^q([0,T];B)\cap L^1_{loc}([0,T];X)$  for $1<q\leq \infty.$  If $\frac{\partial E}{\partial t}:=\left\{ \frac{\partial f}{\partial t}: f \in E \right\}$ is bounded in $L^1_{loc}([0,T];Y),$ then $E$ is relatively compact in $L^p([0,T]; B)$  for all $p<q.$
\end{theorem}

\begin{corollary}\label{cor:L^2-L^2 conver}
Under the same assumptions of Theorem \ref{th:L^p estimates} and supposing $\varrho_0 \in L^2(\Omega)$, then the curves $\varrho^{\tau}_{t}$ and $\varrho^{\tau,\varepsilon}_{t}$  converge strongly to the solution of the Fokker-Planck equation in $L^2([0,T];L^2(\Omega))$ as $\tau \rightarrow 0$. 
\end{corollary}
\begin{proof}
Let us consider Theorem \ref{Simon} for  the Banach spaces $X=H^1(\Omega), B=L^2(\Omega), Y=\mathcal{X}'$ and the indexes $p=2, q=+\infty.$ By the result of Proposition \ref{p^tau,eps}, for each $\varepsilon \in (0,1),$ the family of curves $E:=(\varrho^{\tau,\varepsilon}_{t})_{\tau>0}$  satisfies Theorem \ref{Simon}, thus it is relatively compact in $L^2([0,T];L^2(\Omega)).$ By the first part of Proposition \ref{p^tau,eps}, the curve $\varrho^{\tau,\varepsilon}_{t}$ converges to the unique solution of (\ref{F-P}) in ${W}_2$ distance. Both the strong convergence in $L^2([0,T];L^2(\Omega))$ and the convergence in ${W}_2$ imply the weak convergence as measures in space-time, so we can conclude that the whole family of curves $E=(\varrho^{\tau,\varepsilon}_{t})_{\tau>0}$ converges strongly to this solution in $L^2([0,T];L^2(\Omega)).$\\
Let $\varrho_t$   be the unique solution of (\ref{F-P}). By $\lim_{\tau \rightarrow 0}||\varrho-\varrho^{\tau,\varepsilon}||_{L^2([0,T];L^2(\Omega))}=0$ and (\ref{L^2 dist p^tau and p^tauep}), we have
\begin{eqnarray*}
 \lim_{\tau \rightarrow 0}||\varrho-\varrho^{\tau}||_{L^2([0,T];L^2(\Omega))}&\leq&  \lim_{\tau \rightarrow 0}||\varrho-\varrho^{\tau,\varepsilon}||_{L^2([0,T];L^2(\Omega))} + \lim_{\tau \rightarrow 0}||\varrho^{\tau,\varepsilon}-\varrho^{\tau}||_{L^2([0,T];L^2(\Omega))}\\
 &
\leq &\sqrt{\varepsilon}\cdot C(T)||\varrho_{0}||_{L^2(\Omega)}.\end{eqnarray*}
Letting $\varepsilon \rightarrow 0, $ we get $ \lim_{\tau \rightarrow 0}||\varrho-\varrho^{\tau}||_{L^2([0,T];L^2(\Omega))}=0.$
\end{proof}

The strong convergence in $L^2$ that we just obtained is not surprising, and is a necessary preliminary to pass on to the main goal of this section, which is a higher-order strong convergence result.
  \begin{theorem}[\textbf{Main Theorem I}]\label{Main Theorem I}
 Let  $[0,T]$ be a finite interval and  $\Omega$  be a  convex domain of $\mathbb{R}^d$ whose boundary $\partial \Omega$ is Lipschitz continuous. We assume  $V$ is Lipschitz continuous and  the initial data  $\varrho_0 $  is in $ \mathcal{P}(\Omega)\cap L^2(\Omega).$ As $\tau \rightarrow 0,$ the curve $\varrho^{\tau}_{t}$  strongly converges as $\tau\to 0$ to the solution of the Fokker-Planck equation (\ref{F-P}) in $L^2([0,T];H^1(\Omega)).$
 \end{theorem}
 \begin{proof}
 Denote by $\varrho_t$ the unique solution of (\ref{F-P}). We saw in Proposition \ref{continuous L2 derivative} that we have
 \begin{equation}\label{contL2der}
 \displaystyle \int_{\Omega}\varrho_T^2(x)\,dx-\int_{\Omega}\varrho_0^2(x)\,dx= -2\int_0^T \int_{\Omega}|\nabla \varrho_t(x)|^2\,dxdt -2\int_0^T \int_{\Omega} \varrho_t(x) \, \nabla \varrho_t(x) \cdot \nabla V(x)\,dxdt. 
 \end{equation}
  By  the estimate (\ref{eq:flow interchange}), we have 
$$
\displaystyle \int_{\Omega}(\varrho_T^{\tau})^2\,dx-\int_{\Omega}\varrho_0^2\,dx=\int_{\Omega}(\varrho_{N}^{\tau})^2dx-\int_{\Omega}\varrho_0^2\,dx\leq$$ $$ -2 \sum_{k=0}^{N-1}\tau \int_{\Omega}\left(|\nabla \varrho^{\tau}_{k+1}|^2+\varrho^{\tau}_{k+1} \,\nabla \varrho^{\tau}_{k+1}\cdot \nabla V\right)\,dx=
$$ 
 \begin{equation}\label{eq:bound for p^t}
 -2 \int_0^T \int_{\Omega}\left(|\nabla \varrho^{\tau}_{t}|^2+\varrho^{\tau}_{t} \, \nabla \varrho^{\tau}_{t}\cdot \nabla V\right)\,dxdt.
 \end{equation}
 The estimate (\ref{eq:bound for p^t}) implies that
  $$\displaystyle \limsup_{\tau\rightarrow 0}\int_{\Omega}(\varrho_T^{\tau})^2\,dx-\int_{\Omega}\varrho_0^2\,dx \leq -2 \limsup_{\tau \rightarrow 0} \int_0^T \int_{\Omega}\left(|\nabla \varrho^{\tau}_{T}|^2+\varrho^{\tau}_{t} \, \nabla \varrho^{\tau}_{t}\cdot \nabla V\right)\,dxdt.$$
 By Theorem \ref{th:Properties},  we know that for any $ t \in [0,T],$ 
  $${W}_2(\varrho^{\tau}_{t}, \varrho_t)\rightarrow 0$$
    uniformly as $\tau\rightarrow 0.$
The convergence in Wasserstein distance for every $t$ implies 
 \begin{equation}\label{eq:not need}
 \displaystyle \liminf_{\tau \rightarrow 0} ||\varrho^{\tau}_t||_{L^2(\Omega)} \geq ||\varrho_t||_{L^2(\Omega)}.
 \end{equation}
 As a consequence of (\ref{eq:not need}), we have 
 
  \begin{equation}\label{eq:L^2 convergence a.e}
   \limsup_{\tau\rightarrow 0}\int_{\Omega}(\varrho_T^{\tau})^2\, dx \geq \int_{\Omega}(\varrho_T)^2\,dx. 
  \end{equation}
  Since $\varrho^{\tau}_{t}$ tends to $\varrho_{t}$  strongly in  $L^2([0,T];L^2(\Omega))$ and weakly in $L^2([0,T];H^1(\Omega)),$ we have
  \begin{equation}\label{eq: passing limit 1}
  \lim_{\tau \rightarrow 0} \int_0^T \int_{\Omega}\varrho^{\tau}_{t} \, \nabla \varrho^{\tau}_{t}\cdot \nabla V \,dx dt= \int_0^T \int_{\Omega} \varrho_{t} \, \nabla \varrho_{t}\cdot \nabla V\,dxdt
  \end{equation}
  and 
  \begin{equation}\label{eq:passing limit 2}
  \liminf_{\tau \rightarrow 0} \int_0^T \int_{\Omega}|\nabla \varrho^{\tau}_{t}|^2\, dxdt\geq \int_0^T \int_{\Omega}|\nabla \varrho_{t}|^2 \,dxdt.
  \end{equation}
 Equations (\ref{eq:L^2 convergence a.e}) and (\ref{eq: passing limit 1})  show that 
 $$\displaystyle 2 \limsup_{\tau \rightarrow      0} \int_{0}^{T} \int_{\Omega} |\nabla \varrho^{\tau}_{t}|^2 dx dt  \leq \int_{\Omega} \varrho_0^2 dx-\int_{\Omega} \varrho_T^2 dx -2\int_{0}^{T} \int_{\Omega} \varrho_{t} \cdot \nabla \varrho_{t} \cdot \nabla V dxdt.$$
Using \eqref{contL2der} and \eqref{eq:passing limit 2}, we obtain
$$
\lim_{\tau \rightarrow 0} \int_0^T \int_{\Omega}|\nabla \varrho^{\tau}_{t}|^2\, dxdt =\int_0^T \int_{\Omega}|\nabla \varrho_{t}|^2\, dxdt.
$$
Together with the weak $L^2$ convergence of $\nabla\varrho^\tau$ to $\nabla\varrho$ this implies that  $\varrho^{\tau}_t$  converges strongly to $\varrho_t$ in $L^2([0,T];H^1(\Omega)).$
\end{proof}

\section{Sobolev estimates on the JKO scheme}
In this section we want to present important  estimates in Sobolev spaces which were established in \cite{F.DiM}. These type estimates require to consider some important  inequalities in optimal transportation.\\
 \begin{lemma}\label{5GI}[\textbf{Five-gradients inequality}]
 Let $\Omega$ be a bounded, uniformly convex domain with $C^2$ boundary and  $\varrho,g \in W^{1,1}(\Omega)\cap C^{\alpha}(\bar{\Omega}),$ $0<\alpha<1,$ be two strictly positive probability densities. Let $H \in C^2(\mathbb{R}^d)$ be a radially symmetric convex function. Then the following hold
 \begin{multline}\label{eq:five.gr}
 \int_{\Omega} \left( \nabla \varrho \cdot \nabla H(\nabla \varphi)+\nabla g \cdot \nabla H(\nabla \psi) \right)dx\\= \int_{\Omega}\varrho \mathrm{Tr}\{D^2H(\nabla\varphi)\cdot (D^2\varphi)^2 \cdot(I-D^2\varphi)^{-1}\}dx+ \int_{\partial\Omega}\left(\varrho\nabla H(\nabla\varphi)\cdot \vec{n} + g \nabla H(\nabla \psi)\cdot \vec{n}\right) d \mathcal{H}^{d-1}\geq 0
 \end{multline}
 where $(\varphi,\psi)$ is a choice of Kantorovich potentials in the transport from $\varrho$ to $g$.
 \end{lemma}
 \begin{proof}
 Because of our assumptions on the densities, Caffarelli's regularity theory (see \cite{C1}-\cite{C6}) implies that the Kantorovich potentials, $\varphi$ and $\psi,$ are in $C^{2+\alpha}(\bar{\Omega}) ,$ so we can use the integration  by part to write
 \begin{multline*}
 \displaystyle \int_{\Omega} \left( \nabla \varrho \cdot H(\nabla \varphi)+\nabla g \cdot H(\nabla \psi) \right)dx=\\
-\int_{\Omega} \left( \varrho\, \text{div}[\nabla H(\nabla \varphi)] +g \, \text{div}[\nabla H(\nabla \psi)] \right)\,dx+ \int_{\partial \Omega} \left(\varrho\nabla H(\nabla \varphi)\cdot \vec{n}+g \nabla H(\nabla \psi)\cdot \vec{n}\right)\,d\mathcal{H}^{d-1}
 \end{multline*}
 Due to the radial symmetry of $H$ we have
 \begin{equation}\label{eq:rad.symm}
 \nabla H(\nabla \psi)=-\nabla H(-\nabla \psi),
 \end{equation}
 and $\nabla H(z)$ is a positive scalar multiple of $z$ for every vector $z$.  Since the gradients of the Kantorovich potentials $\nabla \varphi$ and $\nabla \psi$ calculated at boundary points are pointing outward $\Omega,$ (as a consequence of the optimal transport map $T(x)=x-\nabla \varphi(x) \in \Omega$ and its inverse $S(x)=x-\nabla \psi \in \Omega$) we see that the boundary term is non-negative:
 \begin{equation}\label{eq:grH*n}
 \nabla H( \nabla \varphi(x)) \cdot \vec{n}(x)\geq 0 \text{   and  }  \nabla H( \nabla \psi(x)) \cdot \vec{n}(x)\geq 0
 \end{equation}
 We now consider the term with the divergence, and we have
 $$ \text{div}[\nabla H(\nabla \varphi)]=\mathrm{Tr} (D^2H(\nabla\varphi)\cdot D^2\varphi),\qquad \text{div}[\nabla H(\nabla \psi)]=\mathrm{Tr} (D^2H(\nabla\psi)\cdot D^2\psi).$$
  We then use $g=T_\#\varrho$ to write
  $$\int_\Omega g\mathrm{Tr} (D^2H(\nabla\psi)\cdot D^2\psi)=\int_\Omega \varrho \mathrm{Tr} (D^2H(\nabla\psi\circ T)\cdot D^2\psi\circ T).$$
  Since $H$ is radially symmetric and hence even, $D^2H$ is also even, and using $\nabla\psi\circ T=-\nabla\varphi$ we can write 
  $$\int_{\Omega} \left( \varrho\, \text{div}[\nabla H(\nabla \varphi)] +g \, \text{div}[\nabla H(\nabla \psi)] \right)\,dx=\int_{\Omega}  \varrho \mathrm{Tr}( D^2H(\varphi)\cdot( D^2\varphi + D^2\psi\circ T)).$$
  We now use the relation $(id-\nabla\psi)\circ T=id$ with $T=id-\nabla\varphi$  to see that we have $D^2\psi\circ T=I-(I-D^2\varphi)^{-1}$ and then $D^2\varphi + D^2\psi\circ T=-(D^2\varphi)^2 (I-D^2\varphi)^{-1},$ which proves the claimed formula. We are only left to see the positivity of the integral on $\Omega$ on the r.h.s. but the integrand is pointwise positive as a consequence of $D^2H\geq 0,$  $(D^2\varphi)^2\geq 0$ and  $I-D^2\varphi\geq 0$ since $\frac{|x|^2}{2}-\varphi$ is convex.
  \end{proof}
  
  The result of previous lemma, that we proved in details, was first proven in \cite{five.grad} , but both in \cite{five.grad}  and in \cite{F.DiM} only the positivity was used, without considering the precise remainder terms. Indeed, even ignoring these terms and just using the positivity of $\int \nabla\varrho\cdot\nabla H(\nabla\varphi)+\nabla g\cdot \nabla H(\nabla \psi)$, it is possible to obtain estimates on the optimizers of the JKO scheme for the Fokker-Planck equation, as it is explained in \cite{F.DiM}.
    
  More precisely, we have the following estimates.
  \begin{proposition}\label{boundsFp}
 Suppose $\Omega$ is a bounded and uniformly convex domain with $C^2$ boundary, and $V \colon \bar{\Omega} \to \mathbb{R} $ is  Lipschitz continuous. Let $\varrho_0 \in W^{1,1}(\Omega)\cap C^{\alpha}(\bar{\Omega}),$ $0<\alpha<1,$  $(\varrho_k^\tau)$ be the sequence obtained in the JKO scheme \eqref{Sequence} and $H$ be a convex radially symmetric function. Let  $(\varphi_k,\psi_k)$ denote the pair of Kantorovich potentials in the transport from $\varrho_{k+1}^\tau$ to $\varrho_k^\tau$ and $T_k$ the corresponding optimal map, i.e. $T_k(x)=x-\nabla\varphi_k(x)$. Then we have
 \begin{eqnarray}
  \int_{\Omega} H(\nabla(\log{\varrho^\tau_k} +V)) d\varrho^\tau_k &\geq&
  \int_{\Omega} H(\nabla(\log{\varrho^\tau_{k+1}} +V)) d\varrho^\tau_{k+1}\notag\\ & +& \int_{\Omega} \nabla H\left( \frac{\nabla \varphi_k}{\tau}\right) \cdot (\nabla V-\nabla V \circ T_k) d\varrho^\tau_{k+1}+R,
\end{eqnarray}
where $R\geq 0$ is the remainder term in the statement of Lemma \ref{5GI}, i.e. 
\begin{eqnarray*} R:= \frac{1}{\tau}\int_{\Omega}\varrho \mathrm{Tr}\{D^2H \left(\frac{\nabla\varphi_k}{\tau}\right)\cdot (D^2\varphi_k)^2 \cdot (I-D^2\varphi_k)^{-1}\}dx +\\ \int_{\partial\Omega}(\varrho^\tau_{k+1}\nabla H\left(\frac{\nabla\varphi_k}{\tau}\right)\cdot \vec{n} + \varrho^\tau_k \nabla H\left(\frac{\nabla \psi_k}{\tau}\right)\cdot\vec{ n} )d \mathcal{H}^{d-1}.
\end{eqnarray*}

In particular, when $H(z)=\frac{|z|^p}{p},$ $V\in V^2(\bar{\Omega})$ and $\lambda \in \mathbb{R}$ such that $D^2V\geq \lambda I$, ignoring the positive terms and using $\nabla H(z)\cdot z=pH(z)$, we obtain
$$F_p(\varrho^\tau_k)\geq (1+p\lambda\tau) F_p(\varrho^\tau_{k+1}).$$
\end{proposition}
\begin{proof}
The proof is exactly the same as in Lemma 5.1 of \cite{F.DiM}, with the only difference that the above statement requires to keep track of the positive remainder terms. We observe that at every step of the JKO scheme the obtained densities are smooth enough to justify the application of Lemma \ref{5GI} of our paper (it is only by chance that the lemmas are both numbered 5.1 in the two different papers).
\end{proof}

We now proceed to some uniform estimates on the minimizers of the JKO scheme and on the corresponding potentials.

  \begin{proposition}\label{propphik}
   Suppose $\Omega$ is a bounded and uniformly convex domain with $C^2$ boundary, and $V \in C^2(\bar{\Omega}).$  Let $(\varrho_k^\tau)$ be the sequence obtained in the JKO scheme \eqref{Sequence}, $(\varphi_k,\psi_k)$ denote the pair of Kantorovich potentials in the transport from $\varrho_{k+1}^\tau$ to $\varrho_k^\tau$ and $T_k$ the corresponding optimal map, i.e. $T_k(x)=x-\nabla\varphi_k(x)$. Suppose that $\varrho_0$ is bounded from below and above by positive constants and denote by $a,b$ two  constants such that $a\leq \log\varrho_0+V\leq b$. Suppose moreover that we have $\varrho_0\in W^{1,p}(\Omega)$ for some exponent $p>d$. Then we have:
 \begin{enumerate}
 \item For each $k$ we have $a\leq \log(\varrho^\tau_k)+V\leq b$. In particular, all $\varrho^\tau_k$ are bounded from below and above by some uniform positive constants.
 \item All the potentials $\varphi_k$ satisfy $||id-T_k||_{L^\infty}=||\nabla\varphi_k||_{L^\infty}\leq C\tau^{1/(d+2)}$.
 \item If $\tau$ is small enough (depending on $V$ and $p$), then the values of $F_p(\varrho^\tau_k)$ are uniformly bounded by a constant depending on $\varrho_0$ and on $T$. In particular, the densities $\varrho^\tau_k$ are bounded in $C^{0,\alpha}$ for $\alpha=1-\frac dp>0$.
 \item All the potentials $\varphi_k$ belong to $C^{2+\alpha}(\bar{\Omega}) $  and  $||\varphi_k||_{C^{2+\alpha}(\bar{\Omega})} $ is bounded by a uniform constant. 
 \item The potentials  $\varphi_k$ also satisfy $||D^2\varphi_k||\leq C\tau^\beta$ for a certain exponent $\beta>0$.
  \end{enumerate}
 
 \end{proposition}

\begin{proof}
\begin{enumerate}
 \item The uniform estimates on $\varrho^\tau$, already cited in \cite{F.DiM}, is the same as in Lemma 2.4 of \cite{Max-Min}.
 \item  Whenever $\mu,\nu$ are two measures in a convex domain $\Omega\subset\R^d$ and $T$ is the corresponding optimal map sending $\mu$ onto $\nu$, if the density of $\mu$ is bounded from below by a constant $c_0>0$, then the following remarkable estimate is proven in \cite{Bouchitte}  :
$$||T-id||_{L^\infty}\leq C(d,c_0) W_2(\mu,\nu)^{2/(d+2)}.$$
 If we combine this with
  \begin{equation*}
  {{W}^2_2(\varrho^\tau_{k+1}, \varrho^\tau_k)} \leq 2\tau \,(J(\varrho^\tau_k)-J(\varrho^\tau_{k+1}))\leq C \tau.
\end{equation*}
and the fact that the density $\varrho^\tau_{k+1}$ is bounded from below by a universal constant, we then have
\begin{equation}\label{unifDphi}
||\nabla \varphi_k ||_{L^{\infty}(\Omega)}^{2+d}=||id-T_k||_{L^{\infty}(\Omega)}^{2+d}\leq CW_2^2(\varrho^\tau_{k+1},\varrho^\tau_k)\leq C\tau.
\end{equation}

 \item First we note that, since $\varrho_0$ is bounded and $W^{1,p}$, we have $F_p(\varrho_0)<+\infty$. Then, Proposition \ref{boundsFp} guarantees that $F_p(\varrho^\tau_k)$ grows at most exponentially (here we use the smallness assumption on $\tau$, as we need $1+ p\lambda\tau>0$, where $\lambda$ is the lower bound for the second derivatives of $V$, which could be negative). We then obtain a uniform bound on $F_p(\varrho^\tau_k)$ and on $||\varrho^\tau_k||_{W^{1,p}}$, as a consequence of the uniform lower bound on $\varrho^\tau_k$. The well-known injection of Sobolev spaces into H\"older spaces gives the rest of the claim.
 \item The bound on $||\varphi_k||_{C^{2+\alpha}(\bar{\Omega})} $ is a consequence of Caffarelli's regularity theory for the Monge-Amp\`ere equation (see \cite{C1}-\cite{C6}), once we have proven that the densities are uniformly bounded from above, from below, and in $C^{0,\alpha}$, when the domain is uniformly convex and $C^{2}$.
 \item If we apply the interpolation inequality  (\ref{eq:Interp.}) to $u=\nabla \varphi$ in the case $\gamma=1+\alpha$ and $\theta=1,$ we get
$$
||\nabla \varphi_k||_{C^{1}(\bar{\Omega})} \leq c ||\nabla \varphi_k ||^{\frac{1}{1+\alpha}}_{C^{1+\alpha}(\bar{\Omega})}\, ||\nabla \varphi_k||^{\frac{\alpha}{1+\alpha}}_{C(\bar{\Omega})}.
$$
Using \eqref{unifDphi} and  the uniform bounds on $ ||\varphi_k||_{C^{2+\alpha}(\bar{\Omega})}$ we obtain the claim with $\beta=\frac{\alpha}{(1+\alpha)(2+d)}.$
 \end{enumerate} 
 \end{proof}
 
 We mentioned an interpolation inequality involving higher-order H\"older norms: here below is a precise statement, whose proofs can be found in \cite[Proposition 1.1.3]{A.L} .
%
\begin{theorem}[\textbf{Interpolation inequality for H\"older continuous functions}]
Let $0<\theta<\gamma$ and  $\Omega$  be a  open set in $\mathbb{R}^d$ with uniformly  $C^{\gamma}$ boundary.  Then there exists a positive constant $c$ depending on $\bar{\Omega}, \theta$ and $\gamma$  such that
\begin{equation}\label{eq:Interp.}
||u||_{C^{\theta}(\bar{\Omega})} \leq c ||u||^{\frac{\theta}{\gamma}}_{C^{\gamma}(\bar{\Omega})}\, ||u||^{1-\frac{\theta}{\gamma}}_{C(\bar{\Omega})} 
\end{equation} 
for all $u \in {C^{\gamma}(\bar{\Omega})}.$

\end{theorem}

\section{Strong convergence for second derivatives}
In this section we are going to prove that the approximate solution which comes from the JKO scheme \eqref{Sequence} convergences strongly to the solution of Fokker-Planck equation in higher order Sobolev spaces.

We assume $\Omega$ is a bounded and uniformly convex domain given by
\begin{equation}\label{omega}
\Omega=\{x \in \mathbb{R}^d: h(x)<0\} \text{  and  }  \partial \Omega=\{x \in \mathbb{R}^d: h(x)=0\},
\end{equation}
where $h $  is a uniformly convex function of $ C^2(\mathbb{R}^d)$ such that $D^2h\geq c Id$ for some positive constant $c.$  We also assume that $0\in h(\mathbb{R}^d) $ is a regular  value of $h,$ i.e. $|\nabla h(x)|\neq 0 $ for all $x \in \{x  \in \mathbb{R}^d: h(x)=0\}.$  Then, it is easy to show that the boundary $\partial \Omega$ is a  regular surface of 
class $C^2$ and its exterior unit  normal vector $\vec{n}$ at $x \in \partial \Omega$ 
  is defined by
\begin{equation*}
\vec{n}(x)=\dfrac{\nabla h(x)}{|\nabla h(x)|}.
 \end{equation*} 

 Let $\{\varrho_k^{\tau}\}_{k=0}^{N}$   be the sequence  defined in the JKO scheme (\ref{Sequence}) and $\varrho^{\tau}_t$ be obtained from it by piecewise constant interpolation.
 
  %
%
We  consider the integral
\begin{equation}
F_2(\varrho^{\tau}_t)=\frac12\int_{\Omega}\left|\frac{\nabla \varrho_t^{\tau}(x)}{\varrho_t^{\tau}(x)}+\nabla V(x)\right|^2 \varrho_t^{\tau}(x) \,dx.
\end{equation}
and we will prove that the approximate curve constructed in the JKO scheme satisfies a discrete analogue of what is shown in Lemma \ref{lem:Fisher F-P}. To get our purpose we mention here the trace theorem in the Sobolev spaces and some properties of Hausdorff measures.
\begin{theorem}[{\textbf{Trace Theorem}}]\label{th:trace theor}
Assume  $p \in [1,+\infty),$ $\Omega$ is a bounded open set of $\mathbb{R}^d$  and $\partial \Omega $  is $C^1.$  Then there exists a bounded operator 
$$
\mathrm{Tr}: W^{1,p}(\Omega) \to L^{p}(\partial \Omega)
$$
such that 

$(i)$ $\mathrm{Tr}u=u\mid_{\partial \Omega}$ for all $u \in W^{1,p}(\Omega)\cap C(\bar{\Omega}).$

$(ii)$ $$
||\mathrm{Tr}u||_{ L^{p}(\partial \Omega)}\leq C||u||_{W^{1,p}(\Omega)}
$$
for each $u \in W^{1,p}(\Omega),$ with constant $C$ depending on only $p$ and $\Omega.$
\end{theorem}
\begin{proof}
See \cite[section 18]{Leoni}.
\end{proof}
\begin{proposition}\label{pr:Hausdorff meas.}
Let $\mathcal{H}^k$ be the $k-$dimensional Hausdorff measure on a metric space $(X,d).$ If $Y\subset X$ is any set and $f,g\colon Y\to X$ satisfy $d(f(y),f(z))\leq C d(g(y),g(z))$ for all $y,z \in Y,$ then $\mathcal{H}^k(f(A))\leq C^k\mathcal{H}^k(g(A))$ for all $A\subset Y.$ 
\end{proposition}
\begin{proof}
See \cite[Proposition 11.18]{Folland}.
\end{proof}
In the following computations,  $\varepsilon(\tau)$  denotes any quantity which depends only on $\bar{\Omega}, \varrho_0, V $ and $\tau$ such that $\varepsilon(\tau) \to 0$ as $\tau \to 0.$
\begin{lemma}\label{lem:Fisher JKO} Under the assumptions in Proposition \ref{propphik}
and the above assumptions on $\Omega$ we have
 $$F_2(\varrho_0)-F_2(\varrho_T^{\tau})\geq 
 \int_0^T \int_{\Omega} |D^2(\log{\varrho_t^{\tau}(x)}+V(x))|^2 \varrho_t^{\tau}(x)\,dxdt+
 $$ 
$$
 \int_{0}^T\int_{\Omega} (\nabla \log{\varrho_t^{\tau}(x)}+\nabla V(x))^T \cdot D^2V(x)\cdot(\nabla \log{\varrho_t^{\tau}(x)}+\nabla V(x)) \varrho_t^{\tau}(x)\,dxdt
 +
 $$
  \begin{equation}\label{eq:Fisher for rho_k}
 \int_0^T\int_{\partial \Omega}\frac{\varrho_t^{\tau}(x)}{|\nabla h(x)|}(\nabla \log{\varrho_t^{\tau}(x)}+\nabla V(x))^T\cdot D^2 h(x)\cdot (\nabla \log{\varrho_{t}^{\tau}(x)}+\nabla V(x))\,d\mathcal{H}^{d-1}dt+\varepsilon(\tau).
\end{equation}
\end{lemma}
\begin{proof}
  Let us apply Proposition \ref{boundsFp} for $\varrho_{k-1}^{\tau}$ and $\varrho_{k}^{\tau}$  $(k \in\{1,...,N\})$ in the case of $H(z)=\frac12 |z|^2,$ then
  \begin{equation*}
  \frac{1}{2}\int_{\Omega} \left|\frac{\nabla{\varrho_{k-1}^{\tau}}}{\varrho_{k-1}^{\tau}} + \nabla V\right|^2 d\varrho_{k-1}^{\tau} -\frac{1}{2} \int_{\Omega} \left|\frac{\nabla{\varrho_{k}^{\tau}}}{\varrho_{k}^{\tau}} + \nabla V\right|^2 d\varrho_{k}^{\tau} \geq 
\end{equation*}
\begin{equation} \label{eq:Ineq1}
  \frac{1}{\tau}\int_{\Omega} \nabla \varphi_k \cdot ( \nabla V-\nabla V \circ T_k) \,d\varrho_k^{\tau}+\frac{1}{\tau}\int_{\Omega} \text{tr} [( D^2\varphi_k)^2 \cdot (Id-D^2 \varphi_k)^{-1}] \,d\varrho_k^{\tau}+
\end{equation}
\begin{equation*}  
  \frac{1}{\tau}\int_{\partial \Omega} \left(\varrho_k^{\tau} \nabla \varphi_k \cdot \vec{n}+\varrho_{k-1}^{\tau} \nabla \psi_k \cdot \vec{n}\right)\,d\mathcal{H}^{d-1}.
  \end{equation*}
 If we sum the inequality above with respect to $k,$ we have 
\begin{equation}\label{eq:Ineq2} 
 F_2(\varrho_0)-F_2(\varrho_T^{\tau})\geq \frac{1}{\tau} \sum_{k=1}^N\int_{\Omega} \nabla \varphi_k \cdot ( \nabla V-\nabla V \circ T_k)\, d\varrho_k^{\tau}+
\end{equation}
\begin{equation*}  
  \frac{1}{\tau} \sum_{k=1}^N\int_{\Omega} \text{tr} [( D^2\varphi_k)^2 \cdot (Id-D^2 \varphi_k)^{-1}]\, d\varrho_k^{\tau}+\frac{1}{\tau} \sum_{k=1}^N \int_{\partial \Omega} \left(\varrho_k^{\tau} \nabla \varphi_k \cdot \vec{n}+\varrho_{k-1}^{\tau} \nabla \psi_k \cdot \vec{n}\right)\,d\mathcal{H}^{d-1}.
  \end{equation*} 
 
  We consider each terms of the inequalities (\ref{eq:Ineq1}) and (\ref{eq:Ineq2}). First, the equality  $T_k=id-\nabla \varphi_k$ and the optimality condition \eqref{Optimality condition} imply
  $$
  \frac{1}{\tau}\nabla \varphi_k \cdot ( \nabla V-\nabla V\circ T_k)=\frac{1}{\tau}\nabla \varphi_k^T \cdot  D^2V\circ \xi \cdot \nabla \varphi_k,$$
  where $\xi(x)$ is, for every point $x$, a suitable point on the line segment connecting $x$ and $T(x)$, obtained from the Taylor expansion of $V$. This can be re-written as
\begin{eqnarray*}
  \frac{1}{\tau}\nabla \varphi_k \cdot ( \nabla V-\nabla V\circ T)&=&\tau  \nabla(\log{\varrho_k^{\tau}}+V)^T\cdot D^2V\cdot \nabla(\log{\varrho_k^{\tau}}+V) \\
  &&+
\tau\nabla(\log{\varrho_k^{\tau}}+V)^T (D^2V\circ \xi-D^2V)\cdot\nabla(\log{\varrho_k^{\tau}}+V).
\end{eqnarray*}
Since $\xi $ satisfies 
    $$
  |\xi-x|\leq|T(x)-x|=|\nabla \varphi_k(x)|\leq ||\nabla \varphi_k||_{L^{\infty}(\Omega)} \leq C \, {\tau}^{\frac{1}{2+d}},
  $$
   the uniform continuity of $D^2V$  implies  
   $
   |D^2V\circ\xi -D^2V|=\varepsilon(\tau),
   $ while the uniform bounds on $||\log{\varrho_k^{\tau}}+V||_{W^{1,p}}$  with respect to $k$ and $\tau$  show that
  $$\frac{1}{\tau}\sum_{k=1}^N\int_{\Omega} \nabla \varphi_k \cdot ( \nabla V-\nabla V \circ T_k)\, d\varrho_k^{\tau}= $$
   $$  
  {\tau}\sum_{k=1}^N\int_{\Omega} (\nabla \log{\varrho_k^{\tau}(x)}+\nabla V(x))^T \cdot D^2V(x)\cdot(\nabla \log{\varrho_k^{\tau}(x)}+\nabla V(x)) \varrho_k^{\tau}(x)\,dx
 +\varepsilon(\tau) =$$
   \begin{equation}\label{eq:V part}  
  \int_{0}^T\int_{\Omega} (\nabla \log{\varrho_t^{\tau}(x)}+\nabla V(x))^T \cdot D^2V(x)\cdot(\nabla \log{\varrho_t^{\tau}(x)}+\nabla V(x)) \varrho_t^{\tau}(x)\,dxdt
 +\varepsilon(\tau). 
 \end{equation}
Let us now consider the matrix $D^2 \varphi_k$. Point 5 in Proposition \ref{propphik} insures that $D^2 \varphi_k$ tends uniformly to zero as $\tau\to 0$. This enables  us to write
\begin{eqnarray*}
\frac{1}{\tau}\int_{\Omega} \text{tr} [( D^2\varphi_k)^2 \cdot (Id-D^2 \varphi_k)^{-1}]\, d\varrho_k^{\tau} & = &\frac{(1+\varepsilon(\tau))}{\tau}\int_{\Omega} |D^2 \varphi_k|^2 \,d\varrho_k^{\tau}\notag\\ &=&(1+\varepsilon(\tau)) {\tau}\int_{\Omega} |D^2(\log{\varrho_k^{\tau}}+V)|^2 \,d\varrho_k^{\tau}, 
\end{eqnarray*}
  where we used again the optimality condition (\ref{Optimality condition}). Subsequently, we have
  \begin{equation}\label{eq:trace}
  \frac{1}{\tau} \sum_{k=1}^N\int_{\Omega} \text{tr} [( D^2\varphi_k)^2 \cdot (Id-D^2 \varphi_k)^{-1}] \,d\varrho_k^{\tau}=
(1+\varepsilon(\tau))  \int_0^T \int_{\Omega} |D^2(\log{\varrho_t^{\tau}(x)}+V(x))|^2 \varrho_t^{\tau}(x)\,dxdt.
\end{equation}

  We now concentrate on the boundary integrals. Since the optimal transport map $T_k$ sends $\bar{\Omega} $  to itself,   for  any $x \in \partial \Omega, $ we have  $T_k(x)=x-\nabla \varphi_k(x)\in \bar{\Omega}.$ Thus,
  $$
  0\geq h(x-\nabla \varphi_k(x))=h(x)-\nabla h(x)\cdot \nabla \varphi_k(x)+\frac{1}{2}(\nabla \varphi_k(x))^T\cdot D^2 h(\zeta(x))\cdot \nabla \varphi_k(x)
  $$
 for some point $\zeta(x) $ lying in the line connecting $x$ and $T_k(x)=x-\nabla \varphi_k(x).$  If we use $h(x)=0 $ and $\nabla h(x)= |\nabla h(x)|\vec{n}(x)$ 
$$
 |\nabla h(x)|\vec{n}(x)\cdot \nabla \varphi_k(x)=\nabla h(x)\cdot \nabla \varphi_k(x)\geq \frac{1}{2}(\nabla \varphi_k(x))^T\cdot D^2 h(\zeta(x))\cdot \nabla \varphi_k(x).
$$  
If we multiply the inequality above by $\frac{\varrho_k^{\tau}(x)}{\tau|\nabla h(x)|}$ and integrate over $\partial\Omega,$ we get \\
 $$
  \frac{1}{\tau}\int_{\partial \Omega}\varrho_k^{\tau}(x)\nabla \varphi_k(x) \cdot \vec{n}(x) \,d\mathcal{H}^{d-1} \geq \frac{1}{2\tau}  
  \int_{\partial \Omega}\frac{\varrho_k^{\tau}(x)}{|\nabla h(x)|} (\nabla \varphi_k(x))^T\cdot D^2 h(\zeta(x))\cdot \nabla \varphi_k(x)\,d\mathcal{H}^{d-1}=
$$
 \begin{equation}\label{eq:boun.int.1}
   \frac{\tau}{2}  
  \int_{\partial \Omega}\frac{\varrho_k^{\tau}(x)}{|\nabla h(x)|}(\nabla \log{\varrho_k^{\tau}(x)}+\nabla V(x))^T\cdot D^2 h(\zeta(x))\cdot (\nabla \log{\varrho_{k}^{\tau}(x)}+\nabla V(x))\,d\mathcal{H}^{d-1}.
\end{equation}
The uniform continuity of $D^2 h$ in $\bar{\Omega}$ and  the uniform estimates
  $$|\zeta(x)-x|\leq |x-T_k(x)|\leq C \, {\tau}^{\frac{1}{2+d}} $$ 
  show that we have
  \begin{equation}\label{eq:ijrho}
D^2h(\zeta(x))=D^2h(x)+\varepsilon(\tau).
  \end{equation}
As a result of (\ref{eq:ijrho}) and the uniform bounds on $\varrho_{k}^{\tau}$, (\ref{eq:boun.int.1}) equals
  $$
 \frac{\tau}{2}  
  \int_{\partial \Omega}\frac{\varrho_k^{\tau}(x)}{|\nabla h(x)|}(\nabla \log{\varrho_k^{\tau}(x)}+\nabla V(x))^T\cdot D^2 h(x)\cdot (\nabla \log{\varrho_{k}^{\tau}(x)}+\nabla V(x))\,d\mathcal{H}^{d-1}+
 $$\begin{equation}\label{eq:1111111111}\tau \varepsilon(\tau)||\nabla \log{\varrho_k^{\tau}}+\nabla V||^2_{L^2(\partial \Omega)}.
\end{equation}

Since the inverse of the optimal transport map $S_k\colonequals T_k^{-1}$ is defined by $ S_k(x)=x-\nabla \psi_k (x) \in \bar{\Omega},$ similar arguments as we have done above  show that  we have
  \begin{equation}\label{eq:bound.int.2}
   \frac{1}{\tau}\int_{\partial \Omega}\varrho_{k-1}^{\tau}(x)\nabla \psi_{k}(x) \cdot \vec{n}(x) d\mathcal{H}^{d-1} \geq  
   \frac{1}{2\tau}\int_{\partial \Omega}\frac{\varrho_{k-1}^{\tau}(x)}{|\nabla h(x)|}(\nabla \psi_{k}(x))^T\cdot D^2 h(\zeta'(x))\cdot \nabla \psi_{k}(x)\,d\mathcal{H}^{d-1}
\end{equation}
 for some point $\zeta'(x) $ lying in the line connecting $x$ and $T_k(x)=x-\nabla \varphi_k(x)$ and satisfying
 \begin{equation}\label{eq:psi'}
 |\zeta'(x)-x|\leq |x-T_k(x)|\leq C \, {\tau}^{\frac{1}{2+d}}.
  \end{equation}  
By the equalities 
$$
-\nabla \psi_k(x)=S_k(x)-x=S_k(x)-T_k(S_k(x))=\nabla \varphi_k(S_k(x)),
$$
 $$\nabla \varphi_k(S_k(x))=-\tau\nabla \log{\varrho_{k}^{\tau}(S_k(x))}-\tau\nabla V(S_k(x))
 $$
  and the Monge-Amp\'ere equation  
  $$
  \text{det}(DS_k(x)) \varrho_{k}^{\tau}(S_k(x))=\varrho_{k-1}^{\tau}(x),
  $$
   the right hand side of (\ref{eq:bound.int.2}) equals  
\begin{equation}\label{eq:bondint22}
 \frac{\tau}{2}\int_{\partial \Omega}\frac{\text{det}(DS_k(x)) \varrho_{k}^{\tau}(S_k(x))}{|\nabla h(x)|}([\nabla \log{\varrho_{k}^{\tau}}+\nabla V ]\circ S_k(x))^T\cdot D^2 h(\zeta'(x))\cdot  ([\nabla \log{\varrho_{k}^{\tau}}+\nabla V ]\circ S_k(x))\,d\mathcal{H}^{d-1}.
 \end{equation}
 Let  $(S_k)_\sharp\mathcal{H}^{d-1}$ denotes the image measure of $S_k$ on $\partial \Omega$ with respect to the Hausdorff measure defined by  $(S_k)_\sharp\mathcal{H}^{d-1}(A)\colonequals\mathcal{H}^{d-1}(S^{-1}_k(A))=\mathcal{H}^{d-1}(T_k(A))$ for all measurabele (w.r.t $\mathcal{H}^{d-1}$) $ A \subset \partial \Omega.$ Our assumptions provide that the maps $T_k, S_k \colon \partial{\Omega} \to \partial{\Omega} $ are homeomorphisms and hence  $(S_k)_\sharp\mathcal{H}^{d-1}$ is well-defined. With the help of
  this image measure, (\ref{eq:bondint22}) can be written  in the following form
\begin{equation}\label{eq:bondint22 with image meas.}
 \frac{\tau}{2}\int_{\partial \Omega}\frac{ \varrho_{k}^{\tau}(x)}{|\nabla h(T_k(x))|}(\nabla \log{\varrho_{k}^{\tau}}(x)+\nabla V(x) )^T\cdot D^2 h(\zeta'(T_k(x)))\cdot  (\nabla \log{\varrho_{k}^{\tau}}(x)+\nabla V(x))\,d(S_k)_{\sharp}\mathcal{H}^{d-1}.
 \end{equation}
On the one hand, the estimate $|T_k(x)-T_k(y)|\leq \mathrm{Lip}(T_k)\,|x-y|$ for all $x, y \in \partial \Omega$  and Theorem \ref{pr:Hausdorff meas.} show that 
\begin{equation*}
(S_k)_{\sharp}\mathcal{H}^{d-1}(A)=\mathcal{H}^{d-1}(T_k(A))\leq (\mathrm{Lip}(T_k))^{d-1}\mathcal{H}^{d-1}(A)
\end{equation*}
for all $A \subset \partial \Omega.$ On the other hand, the estimate $|x-y|=|S_k(T_k(x))-S_k(T_k(y))|\leq \mathrm{Lip}(S_k)\,|T_k(x)-T_k(y)|$ for all  $x, y \in \partial \Omega$  and  again Theorem \ref{pr:Hausdorff meas.} show that 
\begin{equation*}
\mathcal{H}^{d-1}(A)\leq (\mathrm{Lip}(S_k))^{d-1}\mathcal{H}^{d-1}(T_k(A))= (\mathrm{Lip}(S_k))^{d-1}(S_k)_{\sharp}\mathcal{H}^{d-1}(A).
\end{equation*}
  Therefore, we have 
\begin{equation}\label{eq:im.meas and Haus.meas}
\frac{1}{(\mathrm{Lip}(S_k))^{d-1}} \mathcal{H}^{d-1}(A)\leq (S_k)_\sharp\mathcal{H}^{d-1}(A)\leq  (\mathrm{Lip}(T_k))^{d-1}\mathcal{H}^{d-1}(A)
\end{equation}
for all $A \subset \partial \Omega.$  Using $DT_k(x)=I-D^2\varphi_k(x),$ $DS_k(x)=I-D^2\psi_k(x)=(DT_k)^{-1}\circ S_k$  together with $||D^2\varphi_k||=\varepsilon(\tau)$ we have 
\begin{equation}\label{eq:Lip}
\mathrm{Lip}(S_k)=1+\varepsilon(\tau),\, \text{    }\,  \mathrm{Lip}(T_k)=1+\varepsilon(\tau).
\end{equation}
The estimate  (\ref{eq:im.meas and Haus.meas}) and (\ref{eq:Lip}) show that 
\begin{equation}\label{eq:im.meas=Haus.meas}
(S_k)_{\sharp}\mathcal{H}^{d-1}(A)=(1+\varepsilon(\tau)) \,\mathcal{H}^{d-1}(A)
\end{equation}
for all $A \subset \partial \Omega.$
The regularity of $h$ and point 2 in Proposition \ref{propphik} provide
\begin{equation}\label{eq:grad h}
\frac{1}{|\nabla h(T_k(x))|}=\frac{1}{|\nabla h(x)|}\frac{|\nabla h(x)|}{|\nabla h(T_k(x))|}=(1+\varepsilon(\tau)) \frac{1}{|\nabla h(x)|}
\end{equation}
and 
\begin{equation}\label{eq:hh}
D^2h(\zeta'(T_k(x)))=D^2 h(x)+\varepsilon(\tau).
\end{equation}
By considering (\ref{eq:bondint22 with image meas.}), (\ref{eq:im.meas=Haus.meas}), (\ref{eq:grad h}) and (\ref{eq:hh}), the right hand side of (\ref{eq:bound.int.2}) equals 
$$\frac{\tau}{2}\int_{\partial \Omega}\frac{ \varrho_{k}^{\tau}(x)}{|\nabla h(x)|}(\nabla \log{\varrho_{k}^{\tau}}(x)+\nabla V(x) )^T\cdot D^2 h(x)\cdot  (\nabla \log{\varrho_{k}^{\tau}}(x)+\nabla V(x))\,d\mathcal{H}^{d-1}+$$
\begin{equation}\label{eq:bondint22 with epsilon}
 \tau\varepsilon(\tau)||\nabla (\log{\varrho_{k}^{\tau}}+ V)
 ||^2_{L^2(\partial \Omega)}.
 \end{equation}
 
Since $h$ is convex function, the integrals (\ref{eq:boun.int.1}) and (\ref{eq:bound.int.2}) are  positive. The positivity of these integrals and   the estimates  (\ref{eq:V part}) and (\ref{eq:trace}) provide the inequality
$$
F_2(\varrho_0)-F_2(\varrho_T^{\tau})\geq 
 \int_0^T \int_{\Omega} |D^2(\log{\varrho_t^{\tau}(x)}+V(x))|^2 \varrho_t^{\tau}(x)\,dxdt+
 $$ 
$$
 \int_{0}^T\int_{\Omega} (\nabla \log{\varrho_t^{\tau}(x)}+\nabla V(x))^T \cdot D^2V(x)\cdot(\nabla \log{\varrho_t^{\tau}(x)}+\nabla V(x)) \varrho_t^{\tau}(x)\,dxdt
 $$
We proved in the previous sections that $ \varrho^{\tau}$ is bounded in $L^2([0,T];H^1(\Omega))$, which provides a bound on the  last integral term. This, together with the lower bounds on $\rho^\tau$, implies that $\log{\varrho_t^{\tau}}+V$  is uniformly bounded in  $L^2([0,T]; H^2(\Omega))$. Theorem \ref{th:trace theor} lets us to conclude that   $\nabla(\log{\varrho^{\tau}}+V),$  $i \in\{1,...,d\}, $ is uniformly bounded in   $L^2([0,T]; L^2(\partial \Omega)).$ 

Considering  (\ref{eq:boun.int.1}), (\ref{eq:1111111111}),  (\ref{eq:bound.int.2}), (\ref{eq:bondint22 with epsilon}) and the fact that $ \log{\varrho_{t}^{\tau}}+V$ and $\nabla \log{\varrho_{t}^{\tau}}+\nabla V$ are bounded in $L^2([0,T]; H^2(\Omega))$ and $ L^2([0,T]; L^2(\partial \Omega))$ respectively, we conclude that
\begin{equation*}  
  \frac{1}{\tau}\sum_{k=1}^N\int_{\partial \Omega} \left(\varrho_k^{\tau} \nabla \varphi_k \cdot \vec{n}+\varrho_{k-1}^{\tau} \nabla \psi_k \cdot \vec{n}\right)\,d\mathcal{H}^{d-1}\geq 
  \end{equation*}
$$
 \sum_{k=1}^N {\tau}    \int_{\partial \Omega}\frac{\varrho_k^{\tau}(x)}{|\nabla h(x)|}(\nabla \log{\varrho_k^{\tau}(x)}+\nabla V(x))^T\cdot D^2 h(x)\cdot (\nabla \log{\varrho_{k}^{\tau}(x)}+\nabla V(x))\,d\mathcal{H}^{d-1}+\varepsilon(\tau)=
$$
\begin{equation}\label{eq:boudary}
\int_0^T\int_{\partial \Omega}\frac{\varrho_t^{\tau}(x)}{|\nabla h(x)|}(\nabla \log{\varrho_t^{\tau}(x)}+\nabla V(x))^T\cdot D^2 h(x)\cdot (\nabla \log{\varrho_{t}^{\tau}(x)}+\nabla V(x))\,d\mathcal{H}^{d-1}dt+\varepsilon(\tau).
\end{equation}
 The estimates (\ref{eq:V part}), (\ref{eq:trace}) and (\ref{eq:boudary}) give the desired estimate.
\end{proof}

\begin{lemma}\label{LEMMA}
Let $D$ be bounded Lipschitz domain of $\mathbb{R}^d$, $\{u_k\}_{k=1}^{\infty}$ be a sequence in $L^2([0,T];H^2(D))$  and bounded in $L^{\infty}([0,T]\times D).$ If there exists $u\in L^{\infty}([0,T];C^{2}(D))$ such that $u_k \to u$ strongly in  $L^2([0,T];H^2(D)),$  then, for any $f\in C^2(\mathbb{R}),$ $\{f(u_k)\}_{k=1}^{\infty}$ converges to $f(u)$ in $L^2([0,T];H^2(D)).$
\end{lemma}
\begin{proof} Using our assumptions and the Gagliardo-Nirenberg inequality (see \cite[section 9]{Brezis}) 
\begin{equation}\label{eq:G-N}
||\nabla v||^4_{L^4(D)}\leq C ||v||^2_{H^2(D)}||v||^2_{L^{\infty}(D)}
\end{equation} 
applied to $v=u_k-u$ we obtain $u_k\to u$ in $L^4([0,T];W^{1,4}(D))$. A simple computation shows 
$$D^2 (f(u_k))=f'(u_k)D^2u_k+f''(u_k)\nabla u_k\otimes \nabla u_k$$
and the $L^2$ convergence of this matrix-valued function to $D^2(f(u))=f'(u)D^2 u+f''(u)\nabla u\otimes \nabla u$ is due to the following facts:
\begin{itemize}
\item $D^2u_k\to D^2 u$ in $L^2([0,T]\times D)$;
\item $\nabla u_k\otimes \nabla u_k\to \nabla u\otimes \nabla u$ in $L^2([0,T]\times D)$;
\item both $f''(u_k)$ and $f'(u_k)$ converge a.e. (to $f''(u)$ and $f'(u)$, respectively) as a consequence of the convergence of $u_k$ to $u$; moreover, these terms are bounded in $L^\infty$ as a consequence of the regularity of $f$ and of the $L^\infty$ bound on $u_k$.\qedhere
\end{itemize}
\end{proof}
We are now ready to prove our main theorem.
\begin{theorem}[\textbf{Main Theorem II}]\label{Main Theorem II}
Suppose $0<T< +\infty,$   $\Omega$ is a bounded and uniformly convex domain  given by (\ref{omega}). Let  $V \in C^{2}(\bar{\Omega}),$  $\varrho_0 \in W^{1,p}({\Omega}) $  for $p>d, $  $\lambda\leq \varrho_0 \leq \Lambda$ for some strictly positive constants $\lambda,\Lambda $ and $\varrho$ be the solution of the Fokker-Planck equation (\ref{F-P}). Then,  ${\varrho^{\tau}_{t}}\to {\varrho}$ strongly in   $L^2([0,T];H^2(\Omega))$ as $\tau \rightarrow 0.$ 
\end{theorem}
\begin{proof}
In  Lemma \ref{lem:Fisher JKO} and its proof,  we get the estimate (\ref{eq:Fisher for rho_k})  and showed that $\log{\varrho_t^{\tau}}+V$ is uniformly bounded in $L^2([0,T]; H^2(\Omega))$ with respect to $\tau.$  Since the space $L^2([0,T]; H^2(\Omega))$ is reflexive and $\varrho_t^{\tau}$ converges strongly in $L^2([0,T]; H^1(\Omega))$ to $\varrho$ (see Theorem \ref{Main Theorem I}), we get $\log{\varrho_t^{\tau}}+V$ converges weakly to $\log{\varrho}+V$ in $ L^2([0,T]; H^2(\Omega)).$ This also implies that $\nabla\log{\varrho_t^{\tau}}+\nabla V$ converges weakly to $\nabla\log{\varrho}+\nabla V$ in $ L^2([0,T]; L^2(\partial\Omega)).$ If $\tau$ tends to zero in (\ref{eq:Fisher for rho_k}), then by the lower semicontinuity of weak convergence with respect to its corresponding norm 
$$
F_2(\varrho_0)-F_2(\varrho_T)\geq F_2(\varrho_0)-\liminf_{\tau \to 0} F_2(\varrho_T^{\tau}))\geq  \limsup_{\tau \to 0}
 \int_0^T \int_{\Omega} |D^2(\log{\varrho_t^{\tau}(x)}+V(x))|^2 \varrho_t^{\tau}(x)\,dxdt+
 $$ 
$$
  \limsup_{\tau \to 0} \int_{0}^T\int_{\Omega} (\nabla \log{\varrho_t^{\tau}(x)}+\nabla V(x))^T \cdot D^2V(x)\cdot(\nabla \log{\varrho_t^{\tau}(x)}+\nabla V(x)) \varrho_t^{\tau}(x)\,dxdt
 +
 $$
  $$
  \limsup_{\tau \to 0} \int_0^T\int_{\partial \Omega}\frac{\varrho_t^{\tau}(x)}{|\nabla h(x)|}(\nabla \log{\varrho_t^{\tau}(x)}+\nabla V(x))^T\cdot D^2 h(x)\cdot (\nabla \log{\varrho_{t}^{\tau}(x)}+\nabla V(x))\,d\mathcal{H}^{d-1}dt \geq $$
  $$ \limsup_{\tau \to 0}
 \int_0^T \int_{\Omega} |D^2(\log{\varrho_t^{\tau}(x)}+V(x))|^2 \varrho_t^{\tau}(x)\,dxdt+
 $$ 
$$
 \int_{0}^T\int_{\Omega} (\nabla \log{\varrho_t(x)}+\nabla V(x))^T \cdot D^2V(x)\cdot(\nabla \log{\varrho_t(x)}+\nabla V(x)) \varrho_t(x)\,dxdt
 +
 $$ 
 $$
  \int_0^T\int_{\partial \Omega}\frac{\varrho_t(x)}{|\nabla h(x)|}(\nabla \log{\varrho_t(x)}+\nabla V(x))^T\cdot D^2 h(x)\cdot (\nabla \log{\varrho_{t}(x)}+\nabla V(x))\,d\mathcal{H}^{d-1}dt
$$
The estimate above  and Corollary \ref{coroF2rho} imply
$$ 
 \int_0^T \int_{\Omega} |D^2(\log{\varrho_t(x)}+V(x))|^2 \varrho_t(x)\,dxdt\geq \limsup_{\tau \to 0}
 \int_0^T \int_{\Omega} |D^2(\log{\varrho_t^{\tau}(x)}+V(x))|^2 \varrho_t^{\tau}(x)\,dxdt.
 $$
 Because of weak convergence of $\log{\varrho^{\tau}}+V$ to $\log{\varrho}+V$ in $L^2([0,T];H^2(\Omega))$ we have  $$\liminf_{\tau \to 0}
 \int_0^T \int_{\Omega} |D^2(\log{\varrho_t^{\tau}(x)}+V(x))|^2 \varrho_t^{\tau}(x)\,dxdt \geq \int_0^T \int_{\Omega} |D^2(\log{\varrho_t(x)}+V(x))|^2 \varrho_t(x)\,dxdt.$$
 Therefore, we have
 \begin{equation}\label{eq:lim}\lim_{\tau \to 0}\int_0^T \int_{\Omega} |D^2(\log{\varrho_t^{\tau}(x)}+V(x))|^2 \varrho_t^{\tau}(x)\,dxdt = \int_0^T \int_{\Omega} |D^2(\log{\varrho_t(x)}+V(x))|^2 \varrho_t(x)\,dxdt.
 \end{equation}
 The limit above and the upper and lower boundson $ \varrho^\tau$ show that  $D^2\log{\varrho^{\tau}}$ is bounded in $L^2([0,T];L^2(\Omega))$.  Since  $L^2([0,T];L^2(\Omega))$ is reflexive and $\varrho^\tau$ converges strongly in $L^2([0,T]; H^1(\Omega))$ to $\varrho,$ then $D^2\log{\varrho^{\tau}}$ converges weakly to $D^2\log{\varrho}$ in $L^2([0,T];L^2(\Omega)).$ By adding $D^2V$ and multiplying times $\sqrt{\varrho^\tau}$, which converges a.e. to $\sqrt{\varrho}$ and is bounded by a constant, we also have weak convergence in  $L^2([0,T];L^2(\Omega))$ of $\sqrt{\varrho^\tau}D^2(\log\varrho^\tau+V)$ to $\sqrt{\varrho}D^2(\log\varrho+V)$. Yet, this convergence becomes strong because of the convergence of the norm in \eqref{eq:lim}. We can then multiply times $(\varrho^\tau)^{-1/2}$ and subtract $D^2V$ and obtain strong convergence in $L^2([0,T];H^2(\Omega))$ for $\log\varrho^\tau$ to $\log\varrho$.  

We then apply  Lemma \ref{LEMMA} to obtain ${\varrho^{\tau}}\to {\varrho} $ in $L^2([0,T];H^2(\Omega)).$  \end{proof}



\end{document}